\documentclass[11pt]{amsart}
\usepackage{amsfonts,newlfont,latexsym}
\usepackage{color}

\usepackage{amssymb,amsmath,amscd,amsthm,pb-diagram}
\usepackage{hyperref}

     \theoremstyle{plain}
     \newtheorem{Theorem}{Theorem}[section]
     \newtheorem{theorem}{Theorem}[Theorem]
     \newtheorem{lemma}[Theorem]{Lemma}

     \newtheorem{Lemma}[Theorem]{Lemma}
     \newtheorem{corollary}[Theorem]{Corollary}
     \newtheorem{proposition}[Theorem]{Proposition}
     
     \newtheorem{remark}[Theorem]{Remark}

\theoremstyle{definition}
\newtheorem{definition}[Theorem]{Definition}
\newtheorem{question}[Theorem]{Question}

\newcommand{\R}{\mathbb R}
\newcommand{\Rk}{\mathbb R^k}

\newcommand{\Z}{\mathbb Z}
\newcommand{\Zk}{\mathbb Z^k}
\newcommand{\Q}{\mathbb Q}
\newcommand{\T}{\mathbb T}

\newcommand{\Ci}{C^{\infty}}

\def \ta{{\alpha}}
\def \tr{{\rho}}
\def \tp{{\phi}}
\def \tm{{\mu}}
\def \tl{{\lambda}}

\def \a{\alpha}

\def \G{\Gamma}

\def \e{\varepsilon}

\def \A{(\EuScript{A})}
\def \M{{M}}

\def \dim{\mbox{dim}\,}
\def \d{\mbox{dist}\,}

\def \ker{\mbox{ker}}

\def \d{\mbox{dist}}

\def \A{\cal A}
\def \Rk {{\mathbb R}^k}
\def \rk {{\mathbb R}^k}

\def \T {{\mathbb T}}

\def \R{{\mathbb R}}
\def \Z{{\mathbb Z}}
\def \Zk{{\mathbb Z} ^k}
\def \Q{{\mathbb Q}}

\def \be{{\bar E}}

\def \w{{\cal W}}
\def \ws{{\cal W} ^s}

\def \ci{C^{\infty}}

\definecolor{rjs}{rgb}{.9,0.0,.7}

\def \colb{}

\DeclareMathOperator{\colim}{{colim}}
\DeclareMathOperator{\Hom}{{Hom}}

\begin{document}
\author[David Fisher, Boris Kalinin,  Ralf Spatzier]
{David Fisher, Boris Kalinin,  Ralf Spatzier $^{\ast }$\\ \\
(with an Appendix by James F. Davis $^{\ast \ast}$)}

\title[Global Rigidity of Higher Rank Anosov Actions on Tori and Nilmanifolds]
{Global Rigidity of Higher Rank Anosov Actions on Tori and Nilmanifolds}

\thanks{$^{\ast }$ Supported in part by NSF grants DMS-0643546, DMS-1101150 and  DMS-0906085  }
\thanks{$^{\ast \ast}$ Supported in part by NSF grant DMS-1210991}

\address{Department of Mathematics,
Indiana University, Bloomington, IN 47405}

\email{fisherdm@indiana.edu}

 \address{Department of Mathematics,
  Pennsylvania State University, University Park, PA 16802}

 \email{bvk102@psu.edu}

\address{Department of Mathematics, University of Michigan, Ann Arbor,
MI 48109.}

\email{spatzier@umich.edu}

\address{Department of Mathematics,
Indiana University, Bloomington, IN 47405}

\email{jfdavis@indiana.edu}

\date{\today}

\maketitle

\begin{abstract}

We show that sufficiently irreducible Anosov actions of higher rank abelian
groups on tori and nilmanifolds are $\ci$-conjugate to affine actions.

\end{abstract}

\section{Introduction}

An  Anosov diffeomorphism $f$ on a torus $\T^n$ is {\em affine} if $f$ lifts
to an affine map on $\R ^n$.  By a classical result of Franks and Manning,
any Anosov diffeomorphism $g$ on $\T^n$ is topologically conjugate to
an affine Anosov diffeomorphism. More precisely, there is a homeomorphism
 $\phi: \T^n \mapsto \T^n$ such that $f = \phi \circ g \circ \phi ^{-1}$ is an
 affine Anosov diffeomorphism. We call $\phi$ the Franks-Manning conjugacy.  The  linear part of  $f$ is  the map induced by $g$ on $H_1 (\T^n)$.

Anosov diffeomorphisms are rarely $C^1$-conjugate to affine ones.
For example, one can perturb a linear Anosov diffeomorphism locally
around a fixed point $p$ to  change the conjugacy class of the derivative at $p$.  The resulting diffeomorphism will still be Anosov but cannot be $C^1$-conjugate to its linearization. The situation is radically different for  $\Z^k$-actions with many Anosov diffeomorphisms.  In other words, Anosov diffeomorphisms
rarely commute with other Anosov diffeomorphisms.

It follows easily from the result for a single Anosov diffeomorphism that
an Anosov $\Z^k$-action $\alpha$ on $\T^n$ is topologically conjugate
to a $\Z^k$-action by affine Anosov diffeomorphisms. We  call this action
the {\em linearization} of $\alpha$ and denote it by $\rho$. Again, for any
$a \in \Z^k$ the linear part of $\rho (a)$ is the map induced by $\a (a)$
on $H_1 (\T^n)$. The  logarithms of the moduli of the eigenvalues of these
linear parts define additive maps $\lambda _i: \Z^k \mapsto \R$, which extend
to linear functionals on $\R^k$.  A {\em Weyl chamber } of $\rho$ is a
connected component of $\R^k - \cup _i \ker \lambda _i$.

\begin{Theorem}\label{theorem:main} Let $\alpha $ be a
$\ci$-action of $\Z ^k$, $k \geq 2$, on a torus $\mathbb T^n$ and let
$\rho$ be its linearization. Suppose that there is a $\Z^2$ subgroup of
$\Z ^k$ such that $\rho (a)$ is ergodic for every nonzero $a \in \Z^2$.
Further assume that there is an Anosov element for $\alpha$ in each
Weyl chamber  of $\rho$. Then $\alpha$ is $\ci$-conjugate to $\rho$.
\end{Theorem}

Furthermore, for a linear $\Z^k$-action on
$\mathbb T^n$ having a $\Z^2$ subgroup acting by ergodic elements
is equivalent to several other  properties,
in particular to being genuinely higher rank \cite{St}. A linear $\Z^k$-action
is called {\em genuinely higher rank} if for all finite index subgroups $Z$ of
$\Z^k$, no quotient of the $Z$-action factors through a finite extension of
$\Z$. Hence we obtain the following corollary.

\begin{corollary}
\label{corollary:genuine}
Let $\alpha $ be a $\ci$-action of $\Z ^k$, $k \geq 2$, on a
torus $\mathbb T^n$. Suppose that the linearization  $\rho$ of $\alpha$
is genuinely higher rank. Further assume that there is an Anosov element
for $\alpha$ in each Weyl chamber of $\rho$. Then $\alpha$ is
$\ci$-conjugate to $\rho$.
\end{corollary}

We can define Weyl chambers for the action $\alpha$ itself. In fact these Weyl chambers will turn out to be the same for $\alpha$ and $\rho$. Hence
existence of Anosov elements for $\alpha$ in every Weyl chamber of
$\rho$ is equivalent to existence of Anosov elements for $\alpha$ in
every Weyl chamber of $\a$.

We refer to our paper \cite{FKS} for a brief survey of  other results and methods in the classification of higher rank Anosov actions.
Our global rigidity results above are optimal except that we require an Anosov element in every Weyl chamber.   Rodriguez Hertz in \cite{RH} classifies higher rank actions on tori  assuming only one Anosov element.  However, his work requires multiple additional hypotheses such as bunching conditions and low dimensionality of coarse Lyapunov spaces.   In particular, the hypotheses in
\cite{RH} require that the rank of the acting group has to grow linearly with the dimension of the torus.  It is a conjecture due to Katok and the third author
that global rigidity holds assuming  $\alpha$ has one Anosov element.
We discuss this conjecture in more detail at the end of this introduction.

Let us briefly describe our proof which crucially uses the Franks-Manning conjugacy $\phi$ for some Anosov element of the action. As we noted,
$\phi$ also conjugates any commuting diffeomorphism to an affine map.
In consequence, each element of the action gives a functional equation
for $\phi$. This yields explicit series representations for its projection
$\phi_V$ to any generalized joint eigenspace $V$ of $\rho$. The existence
of Anosov elements of $\a$ in every Weyl chamber allows to define
{\colb coarse Lyapunov foliations as finest nontrivial intersections of  stable
and unstable foliations of Anosov elements. Since the latter are continuous, so are the coarse Lyapunov foliations. It is precisely here that existence of an Anosov element in each Weyl chamber is used.  We then employ the continuity of the coarse Lyapunov foliations  to obtain uniform estimates for contraction and expansion.   Thus  elements close to a Weyl chamber wall act almost isometrically along  suitable coarse Lyapunov foliations, or more precisely,
we can make their exponents in these estimates as close to 0 as we wish, and in particular smaller than the size of the exponent in the  exponential decay we get from exponential mixing.   We  use such elements to study the regularity of $\phi_V$ along each
coarse Lyapunov foliation $\w$. Using exponential mixing for H\"older functions we
show that the partial derivatives along $\w$ exist as distributions dual
to spaces of H\"{o}lder functions.  Then we adapt ideas from a paper by
Rauch and Taylor to show that $\phi$ is smooth. We emphasize that the
rigidity of $\Zk$-actions for $k \geq 2$ is due to the co-existence of (almost) isometric and hyperbolic behavior in the actions. This utterly fails for $\Z$-actions.
}

The paper is organized as follows. We  first explain general definitions, constructions and results for higher rank Anosov actions in Section \ref{preliminaries}.   In Section \ref{franks} we  turn to actions on tori and nilmanifolds, and  use the Franks-Manning conjugacy to derive special properties of such actions.  Most importantly, we will develop uniform growth estimates for elements near the Weyl chamber walls of the action in Section \ref{uniform estimates}.  We then turn  to the case of the torus as it is substantially more elementary than the nilmanifold case.  In Section \ref{exponential}, we establish exponential mixing for $\Z ^{k}$-actions by ergodic affine automorphisms on a torus.  For smooth actions on tori with the standard smooth structure we
prove in Section \ref{soln-cohomology} the existence of partial derivatives
 in all directions as distributions dual to H\"{o}lder functions. This concludes
 the proof for the case of standard tori using the general regularity result
that we establish in Section \ref{fromrauch}. For exotic tori, i.e. manifolds
that are homeomorphic to but not diffeomorphic to tori, in dimensions at
least 5 we can pass to a finite cover with the standard smooth structure.
For dimension 4 we give a special argument in Section \ref{four}.

 {\colb   Finally, we adapt our arguments to nilmanifolds:  Let $N$ be a simply connected nilpotent Lie group.  We call a diffeomorphism of $N$ {\em affine} if it is a composition of an automorphism of $N$ with a left translation by an element of $N$.   If $\Gamma \subset N$ is a   discrete subgroup, we call the quotient $N/\Gamma$  a   {\em nilmanifold}.
 An {\em  infra-nilmanifold} $M$ is a manifold finitely covered by a  nilmanifold.  Diffeomorphisms of $M$ covered by affine diffeomorphisms of $N$ are   again  called {\em affine}.  The Franks-Manning conjugacy theorem generalizes to  infra-nilmanifolds: Suppose $M'$ is a smooth manifold homeomorphic with an
 infra-nilmanifold.  Then every Anosov diffeomorphism  of $M'$ is conjugate to an affine diffeomorphism 
 of $M$ by a homeomorphism $\phi$.  {\colb We  call $\phi$ the Franks-Manning conjugacy.   Given an action $\alpha$ of $\Zk$ on $M'$ which contains an Anosov diffeomorphism, then its Franks-Manning conjugacy   jointly conjugates all $\alpha (a), a \in \Zk$, to affine diffeomorphisms $\rho (a)$.   We call $\rho$ the linearization of $\alpha$. }
 Now we can state our main result for  nilmanifolds:

 \begin{Theorem}\label{theorem:nil} Let $\alpha $ be a
$\ci$-action of $\Z ^k$, $k \geq 2$, on a compact infra-nilmanifold $N/\Gamma$ and let
$\rho$ be its linearization. Suppose that there is a $\Z^2$ subgroup of
$\Z ^k$ such that $\rho (a)$ is ergodic for every nonzero $a \in \Z^2$.
Further assume that there is an Anosov element for $\alpha$ in each
Weyl chamber  of $\rho$. Then $\alpha$ is $\ci$-conjugate to $\rho$.
\end{Theorem}

  Our  main result reduces
 to the case of standard nilmanifolds, i.e. nilmanifolds  with the differentiable structure coming from the ambient Lie group. Indeed, there are no Anosov diffeomorphisms on non-toral nilmanifolds in dimensions 4 or less, and the result
 by J. Davis in the appendix shows that any nilmanifold of dimension
 at least 5 is finitely covered by a standard nilmanifold.

 For standard nilmanifolds we proceed  similarly to  the toral case.   We adapt arguments of Margulis and Qian  \cite{MaQi} to reduce regularity of the conjugacy  to regularity of the solution of a cohomology equation.  The relevant cocycle however  takes values in a nilpotent group, and  is not directly amenable to our approach.  Instead, we consider  suitable factors of the cocycle in various abelian quotients of  the derived series of $N$.  Again we prove regularity of coboundaries for the resulting cocycles by exponential mixing of the $\Z ^k$ action, uniform expansion and contraction of elements close to Weyl chamber walls, and showing existence of derivatives via distributions dual to H\"{o}lder  functions.
 Unlike in the toral case, exponential mixing of actions by affine automorphisms does not follow from  elementary Fourier analysis.  Rather this was established by Gorodnik and the third author in\cite{GorodnikSpatzier}.
 We remark that this approach yields  the first rigidity results for higher rank actions on general nilmanifolds.  Earlier cocycle and local rigidity results, by A. Katok and the third author, were only proved for actions which were higher rank both on the toral factor as well as the fibers (e.g. \cite{KS-IHES}).  There, cocycles were straightened out separately on the base and the fibers.
Exponential mixing of these actions thus allows for a much simpler and direct approach, and is {\colb also} used in   \cite{GorodnikSpatzierII} to prove cocycle rigidity results.

}

 {\it Epilogue}: We conclude this paper with some remarks about the
 conjecture by Katok and Spatzier that genuinely higher rank abelian
 Anosov actions  are smoothly conjugate to affine actions.
Using the arguments of our earlier paper \cite{FKS}, we can show that the conjugacy is always smooth along almost every leaf of each coarse Lyapunov foliation.  However, we have no further evidence in support of this conjecture and in fact have some doubts about its truth. In
\cite{Gogolev-holder} Gogolev constructed a diffeomorphism of a torus which is H\"{o}lder conjugate to an Anosov diffeomorphism but itself is not Anosov. Thus having one
 Anosov element may not imply that most elements are Anosov.
 In \cite{FarrellJones}, Farrell and Jones constructed Anosov diffeomorphisms on exotic tori.  In light of this construction it seems obvious to ask:

\begin{question}
\label{question:exotic}
Are there genuinely higher rank Anosov $\Z^k$ actions on exotic tori?
\end{question}

 As  exotic tori are finitely covered by standard tori, such actions would lift to actions on standard tori.  The latter could not be smoothly equivalent to their linearizations since a  smoothness result for conjugacy would descend to the $C^{0}$ conjugacy between the exotic and standard torus.  Thus such examples would also give counterexamples to the conjecture by Katok and Spatzier even when the  underlying smooth structure on the torus is standard.

 We remark here that the construction in \cite{FarrellJones}, further explained and simplified
in \cite{FarrellGogolev}, does not adapt easily to the case of actions of higher rank abelian groups.
 Indeed because of the delicate  cutting and pasting arguments used in their constructions,
it would be hard to
guarantee that different elements continue to commute.  As a consequence of Theorem \ref{theorem:main},
a positive answer to Question \ref{question:exotic} can only occur for an action where relatively few
elements are Anosov.
Furthermore, by the results in \cite{RH},  a positive answer to Question \ref{question:exotic} seems unlikely if the dynamically defined foliations for the
action have dimensions $1$ or $2$.
The Farrell-Jones construction proceeds by cutting and pasting
exotic spheres into the torus.  This suggests, in order to construct examples for Question \ref{question:exotic}, one would want to glue in the exotic sphere in a manner somehow subordinate to the dynamical
foliations using their high dimension.

We are indebted to  J. Rauch for discussions concerning his result with M. Taylor on  regularity  for distributions.  A strengthening of one of their theorems is fundamental to our approach and multiple discussions with Rauch were a key to our first believing and then proving this result.  We  also thank  A. Gorodnik and J. Conlon for various discussions.  Finally, we are  more than grateful to J. Davis for discussions concerning non-standard smooth structures and for writing the appendix on exotic differentiable structures on nilmanifolds.


\section{Preliminaries}   \label{preliminaries}

Throughout the paper, the smoothness of diffeomorphisms, actions,
and manifolds is assumed to be $\Ci$, even though all definitions and
some of the results can be formulated  in lower regularity.

\subsection{Anosov actions of $\,\Zk$}
\label{Anosov actions} $\;$
\vskip.1cm

Let $a$ be a diffeomorphism of a compact manifold $\M$.
We recall that  $a$ is  {\em Anosov} if there exist a continuous $a$-invariant
decomposition  of the tangent bundle $T\M=E^s_a \oplus E^u_a$ and constants
$K>0$, $\lambda>0$ such that for all $n\in \mathbb N$
\begin{equation} \label{anosov}
  \begin{aligned}
 \| Da^n(v) \| \,\leq\, K e^{- \lambda n} \| v \|&
     \;\text{ for all }\,v \in E^s_a, \\
 \| Da^{-n}(v) \| \,\leq\, K e^{- \lambda n}\| v \|&
     \;\text{ for all }\,v \in E^u_a.
 \end{aligned}
 \end{equation}
The distributions $E_a^s$ and $E_a^u$ are called the {\em stable} and {\em unstable}
distributions of  $a$.
\vskip.2cm

Now we consider a $\Zk$ action $\alpha$ on a compact manifold $\M$
via diffeomorphisms.
The action is called {\em Anosov}$\,$ if there is an element which acts as an
Anosov diffeomorphism. For an element $a$ of the acting group we denote
the corresponding diffeomorphisms by $\a (a)$ or
simply by $a$ if the action is fixed.

The distributions $E_a^s$ and $E_a^u$  are H\"{o}lder continuous and
tangent to the stable and unstable foliations $\w_a^s$ and $\w_a^u$
respectively \cite{HPS}. The leaves of these foliations are $C^\infty$
injectively immersed Euclidean spaces. Locally, the immersions vary
continuously in the $\Ci$ topology.  In general, the distributions $E^s$
and $E^u$ are only H\"older continuous transversally
to the corresponding foliations.

\subsection{Lyapunov exponents and coarse Lyapunov distributions}
\label{Lyapunov} $\;$
\vskip.1cm

First we recall some basic facts from the theory of non-uniform
hyperbolicity for a single diffeomorphism, see for example \cite{BP}.
Let  $a$ be a diffeomorphism of a compact manifold $\M$ preserving an
ergodic  probability measure $\mu$. By Oseledec' Multiplicative Ergodic
Theorem, there exist finitely many numbers $\chi _i $ and an invariant  measurable
splitting of the tangent bundle $T\M = \bigoplus E_i$ on a set of full measure
such that the forward and backward Lyapunov exponents of $v \in  E_i $
are $\chi _i $.  This splitting is called {\em Lyapunov decomposition}.
 We define the stable distribution of $a$ with respect to $\mu$
as $E^-_a= \bigoplus _{\chi _i <0}  E_i$. The subspace $E^-_a(x)$
is tangent $\mu$-a.e. to the stable manifold $W^-_a(x)$.
More generally, given any $\theta <0$ we can define the strong
stable distribution by $E^{\theta}_a= \bigoplus _{\chi _i \le \theta}  E_i$
which is tangent $\mu$-a.e. to the strong stable manifold $W^{\theta}_a(x)$.
$W^{\theta}_a(x)$ is a smoothly immersed Euclidean space. For a
sufficiently small ball $B(x)$, the connected component of
$W^{\theta}_a(x) \cap B(x)$, called {\em local} manifold, can
be characterized by the exponential contraction property:
for any sufficiently small $\e >0$ there exists $C=C(x)$ such that
\begin{equation} \label{loc}
W^{\theta,loc}_a(x) = \{y \in B(x) \, | \; \d(a^n x, a^n y) \le C e^{(\theta +\e)n}
\quad \forall n\in \mathbb N \}.
\end{equation}
The unstable distributions and manifolds are defined similarly.
In general, $E^-_a$ is only measurable and depends on the measure $\mu$.
However, if $a$ is an  Anosov diffeomorphism then $E^-_a$ for any
measure always agrees with the
continuous stable distribution $E^s_a$. Indeed, $E^s_a$ cannot contain
a vector with a nontrivial component in some $E_j$ with $\chi _j \geq 0$
since such a vector does not satisfy \eqref{anosov}.  Hence $E^s_a \subset \bigoplus _{\chi _i <0}  E_i$.  Similarly, the unstable distribution $E^u_a \subset \bigoplus _{\chi _i  > 0}  E_i$.  Since $TM =E^s_a \oplus E^u_a$,
 both inclusions have to be equalities.

Now we consider the case of $\Zk$ actions. Let $\mu$ be an ergodic
probability measure for a $\Zk$ action $\alpha$ on a compact manifold $\M$.
By commutativity, the  Lyapunov decompositions
for  individual elements of $\Zk$ can be refined to a joint invariant splitting for
the action. The following proposition from \cite{KaSa3} describes the
Multiplicative Ergodic Theorem for this case. See  \cite{KaK} for more
details on the Multiplicative Ergodic Theorem and related notions for
higher rank abelian actions.

 \begin{proposition}
 There are finitely many linear functionals $ \chi $ on $\Zk$, a set of
 full measure ${\cal P}$, and an $\alpha$-invariant  measurable splitting
 of the tangent bundle  $T\M = \bigoplus E_{\chi}$ over ${\cal P}$  such that
  for all $a \in \Zk$ and $ v \in E_{\chi}$, the Lyapunov exponent of $v$ is $\chi
  (a)$, i.e.
 $$
   \lim _{n \rightarrow \pm \infty }
   {n}^{-1} \log \| D a ^n  (v) \| = \chi (a),
 $$
  where $\| .. \|$ is a continuous norm on $T\M$.
\end{proposition}

The splitting  $\bigoplus E_{\chi}$ is called the
{\em Lyapunov decomposition},
and the linear functionals $\chi$, extended to linear functionals on $\Rk$,
 are called the {\em Lyapunov exponents} of $\alpha$. The hyperplanes
 $\,\ker \,\chi \subset \Rk$ are called the {\em Lyapunov hyperplanes} or
 {\em Weyl chamber walls}, and the connected components of
 $\,\Rk - \cup _{\chi} \ker \chi $ are called the {\em Weyl chambers} of
 $\alpha$. The elements in the union of the Lyapunov hyperplanes are
 called {\em singular}, and the elements in the union of the Weyl chambers
are called {\em regular}.

Consider  a $\Zk$ action by {\em automorphisms} of a torus
$M=\T^d=\R^d/\Z^d$ or, more generally, a nilmanifold $M=N/\G$,
where $N$ is a simply connected nilpotent Lie group and $\G \subset G$
is a (cocompact) lattice. In this case, the Lyapunov decomposition
is determined by the eigenspaces of the $d \times d$ matrix that
defines the toral automorphism or by the eigenspaces of the induced
automorphism on the Lie algebra of $N$. In particular,
every Lyapunov distribution is smooth and in the toral case integrates
to a linear foliation. The Lyapunov exponents are given by the logarithms
of the moduli of the eigenvalues. Hence they are independent of the
invariant measure and give uniform estimates of expansion and contraction
rates.

In the non-algebraic case, the individual Lyapunov distributions are in
general only measurable and depend on the given measure. This can
be already seen for a single diffeomorphism, even if Anosov. However,
as we observed above, the full stable distribution $E^s_a$ of an Anosov
element $a$ always agrees with $\bigoplus _{\chi  (a) <0}  E_\chi$ on
a set of full measure for any measure.

For higher rank actions, {\em coarse Lyapunov distributions} play a similar
role to the stable and unstable distributions for an Anosov diffeomorphism.
For any Lyapunov functional $\chi$ the coarse Lyapunov distribution is the
direct sum of all Lyapunov spaces with Lyapunov exponents, as functionals,
positively proportional to $\chi$:
$$
   E^{\chi} = \oplus E_{\chi '}, \quad \chi ' = c \, \chi \,\text{ with }\, c>0.
$$

For an algebraic action such a distribution is a finest nontrivial intersection
of the stable distributions of certain Anosov elements of the action.
For nonalgebraic actions, however, this is not a priori clear. It was shown
in \cite[Proposition 2.4]{KaSp} that, in the presence of sufficiently many
Anosov elements, the coarse Lyapunov distributions are well-defined,
continuous, and tangent to foliations with smooth leaves. We quote the
discrete time version \cite[Proposition 2.2]{KaSa4}. We denote the set
of all Anosov elements in $\Zk$  by $\A$.

\begin{proposition} \label{CoarseLyapunov}
Let $\alpha$ be an Anosov action of $\Zk$ and let $\mu$ be an ergodic  probability
measure for $\a$ with full support. Suppose that there exists an Anosov element
in every Weyl chamber defined by $\mu$. Then for each Lyapunov exponent
$\chi $ the coarse Lyapunov distribution can be defined as
$$
  E^{\chi}(p) =  \bigcap _{\{a \in \A\, | \; \chi (a) <0\}} E^s _a(p) \;=
  \bigoplus_ {\{ \chi ' = c \, \chi\,|\;  c>0 \}}  E_{\chi '} (p)
 $$
on the set ${\cal P}$ of full measure where the Lyapunov splitting exist.
Moreover, $E^{\chi}$ is H\"{o}lder continuous, and thus it can
be extended to a H\"{o}lder distribution tangent to the foliation
$\w ^\chi = \bigcap  _{\{a  \in \A \,  \mid \, \chi (a) <0\}} \ws _a$
with uniformly $\Ci$ leaves.
\end{proposition}

Note that ergodic measures with full support always exist if a $\Zk$ action
contains a transitive Anosov element. A natural example is given by the
measure $\mu$ of maximal entropy for such an element, which is unique
\cite[Corollary 20.1.4]{HaKa} and hence is invariant under the whole action.
{\colb We emphasize that it is precisely here where we use the assumption that every Weyl chamber contains an Anosov element.
We will use Proposition \ref{CoarseLyapunov}  in the next section to get uniform estimates for elements close to Weyl chamber walls. }

Since a coarse Lyapunov distribution is defined by a collection of
positively proportional Lyapunov exponents, it can be uniquely identified
with the subset of $\Rk$ where these functionals are positive (resp.
negative). This subset is called the {\em positive (resp. negative)
Lyapunov half-space}. Similarly, a coarse Lyapunov distribution can
be defined with the oriented Lyapunov hyperplane that separates
the corresponding positive and negative Lyapunov half-spaces.

\section{$\Zk$  actions on tori and nilmanifolds and uniform estimates.}
\label{franks}

From now on we consider Anosov $\Zk$ actions on tori and nilmanifolds.  In this section, we explore the special features we obtain thanks to the Franks-Manning conjugacy.   This allows us to control invariant measures, Lyapunov exponents, and even  upper bounds of expansion for  elements close to a Weyl chamber wall (cf. Section \ref{uniform estimates}).

\subsection{Invariant measures and Lyapunov exponents} Let $f$ be an Anosov diffeomorphism of a torus $M= \T^d$ or, more
generally, of a nilmanifold $M=N/ \Gamma$. By the results of Franks
and Manning in \cite{Fr,M}, $f$ is topologically conjugate to an Anosov
automorphism $A : M \to M$, i.e. there exists a homeomorphism
$\phi : M \to M$ such that $A \circ \phi = \phi \circ f$. The conjugacy
$\phi$ is  bi-H\"older, i.e. both $\phi$ and $\phi^{-1}$ are H\"older
continuous with some H\"older exponent $\gamma$.

Now we consider an Anosov $\Zk$ action $\a$ on a nilmanifold $M$.
Fix an Anosov element $a$ for $\a$. Then we have $\phi$ which
conjugates $\a(a)$ to an automorphism $A$. By \cite[Corollary 1]{W}
any homeomorphism of $M$ commuting with $A$ is an affine automorphism.
Hence we conclude that $\phi$ conjugates $\a$ to an action $\rho$
by affine automorphisms. We will call $\rho$ an {\em algebraic action}
and refer to it as the {\em linearization} of $\a$.

Now we describe the preferred invariant measure for $\a$
(cf.  \cite[Remark 1]{KaK07}).
We denote by $\lambda$ the normalized Haar measure on the
nilmanifold $M$. Note that $\lambda$ is invariant under any
affine automorphism of $M$ and is the unique measure of
maximal entropy for any affine Anosov automorphism.

\begin{proposition} \cite[Proposition 2.4]{FKS} \label{mu}
The action $\a$ preserves an absolutely continuous measure
$\mu$ with smooth positive density. Moreover,
$\mu = \phi^{-1} _\ast (\lambda)$ and for any Anosov element
$a\in \Zk$, $\mu$ is the unique measure of maximal entropy for $\a(a)$.
\end{proposition}

In the next proposition we show that the Lyapunov exponents of $(\a, \mu)$
and $(\rho, \lambda)$ are positively proportional and that the corresponding
coarse Lyapunov foliations are mapped into each other by the conjugacy
$\phi$. From now on, instead of indexing a coarse Lyapunov foliations by a
representative of the class of positively proportional Lyapunov functionals,
we index them numerically, i.e. we write $\w^i$ instead of $\w^{\chi}$,
 implicitly identifying the finite collection of equivalence classes of Lyapunov
exponents with a finite set of integers.

\begin{proposition} \label{properties}
 Assume there is an Anosov element in every  Weyl chamber. Then

\begin{enumerate}

\item The Lyapunov exponents of $(\ta, \tm)$ and $(\tr, \tl)$ are positively
proportional, and thus the Lyapunov hyperplanes and Weyl chambers are
the same.

\item For any coarse Lyapunov foliation $\w^i_\ta$ of $\ta$
$$
\tp (\w^i_\ta ) = \w^i_\tr ,
$$
where $\w^i_\ta$ is the corresponding coarse Lyapunov foliation for $\tr$.

\end{enumerate}

\end{proposition}

{\bf Remark.} In fact, one can show that (1) holds for Lyapunov
exponents and coarse Lyapunov foliations of $(\a, \nu)$ for any
$\a$-invariant measure $\nu$ so, in particular, the Lyapunov
exponents of all $\a$-invariant measures are positively proportional
and the coarse Lyapunov splittings are consistent with the continuous
one defined in Proposition \ref{CoarseLyapunov}.

{\bf Remark.} We do not claim  at this point that the Lyapunov exponents
of $(\ta, \tm)$ and $(\tr, \tl)$ (or of different invariant measures for $\ta$)
are equal. Of course, if $\ta$ is shown to be smoothly conjugate to
$\tr$ then this is true a posteriori.

\begin{proof} The proposition is the discrete time analogue of
 \cite[Proposition 2.5]{FKS}. We include the proof for the sake of
 completeness.
 First we observe that the  conjugacy $\tp$ maps the stable
manifolds of $\ta$ to those of $\tr$. More precisely, for any $a \in \Zk$
and any for $\mu$-a.e. $x \in M$ we have
\begin{equation} \label{phimapsW}
\tp (W^-_{\ta (a)}(x)) = W^-_{\tr (a)}(\tp(x)).
 \end{equation}
Indeed, it suffices to establish this for local manifolds, which are
characterized by the exponential contraction as in \eqref{loc}.
Since $\tp$ is bi-H\"older, it preserves the  property that $\d (x_n, y_n)$
decays exponentially, which implies \eqref{phimapsW}.
In particular, for any Anosov $a \in \Zk$ and any $x \in M$
we have $\tp (W^s_{\ta (a)}(x)) = W^s_{\tr (a)}(\tp(x))$.
Hence the formula for $\w^i_\ta$ given in Proposition
\ref{CoarseLyapunov} implies (2) once we establish (1).

To establish (1)  it suffices to show that the oriented Lyapunov hyperplanes
of $(\ta, \tm)$ and $(\tr, \tl)$ are the same. Suppose that an oriented Lyapunov
hyperplane $L$ of one action, say $\ta$, is not an oriented Lyapunov hyperplane of the other action $\tr$.  Then we can take $\Zk$ elements
$a \in L^+$ and $b \in L^-$ which are
not separated by any Lyapunov hyperplane of either action other than $L$.
Then, $E^-_{\ta (b)} = E^-_{\ta (a)} \oplus E$, where $E$ is the
coarse Lyapunov distribution of $\ta$ corresponding to $L$.
Similarly, since we assumed that $L^+$ is  not a positive Lyapunov
 half-space for $\tr$, we have  $E^-_{\tr (b)} \subseteq E^-_{\tr (a)}$.
 We conclude that

$$
W^-_{\ta (a)}  \subsetneq W^-_{\ta (b)} \quad \text{ but } \quad
W^-_{\tr (a)}  \supseteq W^-_{\tr (b)},
$$
which contradicts \eqref{phimapsW} since $\tp$ is a homeomorphism.
\end{proof}


\subsection{Uniform estimates for elements near Lyapunov hyperplane}
\label{uniform estimates}

The uniform estimates proved in this section will play a crucial role in the proof of the main theorem.  They give us upper bounds with small exponents for the expansion in certain directions for elements close to the Weyl chamber walls. This almost isometric behavior together with strong hyperbolic behavior in other directions and exponential mixing will force the convergence of suitable series as distributions.

We first address estimates for the first derivatives of these elements.
We fix a positive Lyapunov half-space $L^+\subset \rk$ and the corresponding Lyapunov hyperplane $L$. We denote the corresponding
coarse Lyapunov distributions for $\ta$ and $\tr$ by $E$ and $\be$
respectively. Recall that $\gamma >0$ denotes a H\"older exponent
of $\tp$ and $\tp ^{-1}$.

\begin{lemma} \label{upper est}
For a given coarse Lyapunov distribution $E$ of $\a$ there exist
linear functionals $\chi_m$ and $\chi_M$ on $\rk$ positive on the
Lyapunov half-space $L^+$ corresponding to $E$ such that for
any invariant ergodic measure $\nu$ of
$\ta (b)$ we have
$$
\chi_m (b) \le \chi_\nu (b) \le \chi_M (b) \qquad \forall  \; b \in L^+ \cap \Zk
$$
where $\chi_\nu (b)$ is any Lyapunov exponent of $(\ta (b), \nu)$ corresponding to the distribution $E$. Equivalently, we have
$\chi_M (c) \le \chi_\nu (c) \le \chi_m (c) $ for all  $\; c \in L^- \cap \Zk$
\end{lemma}

\begin{proof}
The Lyapunov exponents of $\rho$ corresponding to $\be$
are functionals positive on $L^+$. Let $\bar \chi_m$ and $\bar \chi_M$
be the ones smallest and the largest on $L^+$. We will show that
$\chi_m = \gamma \bar \chi_m$ and $\chi_M = \gamma^{-1} \bar \chi_M$
satisfy the conclusion of the lemma.

First we will prove the second inequality, which is slightly easier.
Suppose that $\chi_\nu (b) > \chi_M(b)$ for some
Lyapunov exponent of $(\ta (b), \nu)$ corresponding to the distribution $E$.
Let $E'$ be the distribution spanned by the Lyapunov subspaces of
$(\ta (b), \nu)$ corresponding to Lyapunov exponents greater than
$\chi_M(b)+\e$. Then, for some $\e >0$, $E'$ has nonzero intersection
with the distribution $E$. The strong unstable distribution $E'(x)$
is tangent for $\nu$-a.e. $x$ to the corresponding local strong unstable
manifold $W'(x)$. Hence the intersection $F(x)$ of $W'(x)$ with the leaf
$W(x)$ of the coarse Lyapunov foliation corresponding to $E$ is a
submanifold of positive dimension. We take $y \in F(x)$ and denote
$y_n=\ta(-nb)(y)$ and $x_n=\ta(-nb)(x)$. Then $x_n$ and $y_n$ converge
 exponentially with the rate at least $\chi_M(b)+\e$. Since the conjugacy $\tp$
 is $\gamma$ bi-H\"older it is easy to see that
$$\d (\tp (x_n),\tp (y_n)) = \d (\tr (-nb)(x),\tr (-nb)(y))$$
decreases at a rate faster than $\gamma \, \chi_M(b)$.
But this is impossible since $\tp$ maps $W(x)$ to $\bar W(\phi(x))$,
the leaf of corresponding Lyapunov foliation of $\rho$, which
is contracted by $\tr (-b)$ at a rate at most
$\bar \chi_M (b) =\gamma \, \chi_M(b)$.

The first inequality can be established similarly. Suppose that
$\chi_\nu (b) < \chi_m (b)$ for some Lyapunov exponent of $(\ta (b), \nu)$ corresponding to the distribution $E$. Let $E''\subset E$ be the Lyapunov
distribution corresponding to this exponent. We cannot assert that $E''$ is
tangent to an invariant foliation, so we consider a curve $l$ tangent to
a vector $0\ne v\in E''(x)$ for some $\nu$-typical $x$. Then the exponent
of $v$ with respect to $\a(-b)$ is $-\chi_\nu (b)$. However, since
$\phi (l) \subset \bar W(\phi(x))$, we can obtain as above
that $l$ is contracted by $\a(-b)$ at the rate at least $\chi_m(b)$.
It is easy to see that this is impossible.
\end{proof}

\begin{proposition} \label{e-slow}
Let $E$ be a coarse Lyapunov distribution and $L^+ \subset \rk$ be
the corresponding  Lyapunov half-space for $\a$.
Then for any element $b \in L^+$ any $\e>0$ there
exists $C=C(b,\e)$ such that
\begin{equation}  \label {e-est}
C^{-1} e^{(\chi_m - \e) n} \| v \| \le  \| D (\ta (nb)) v \| \le  C e^{(\chi_M +\e) n} \| v \|
\quad \text{for all } v \in E, n\in \mathbb N,
\end{equation}
where $\chi_m$ and $\chi_M$ are as in Lemma \ref{upper est}
\end{proposition}

\begin{proof}
In the proof we will abbreviate $\ta (b)$ to $b$.
Consider functions $a_n (x) = \log \| D b^n |_E(x) \|$, $n \in \mathbb N$.
Since the distribution $E$ is continuous, so are the functions $a_n$.
The sequence $a_n$ is subadditive, i.e.
$a_{n+k} (x) \le a_n (b^k(x)) + a_k (x)$.
The Subadditive and Multiplicative Ergodic Theorems imply that for every
$b$-invariant ergodic measure $\nu$ the limit
$\lim _{n\to \infty} \, {a_n(x)/n}$ exists for $\nu$-a.e.$\,x$ and equals
the largest Lyapunov exponent of $(b, \nu)$ on the distribution $E$.
The latter is at most $\chi_M(b)$ by Lemma \ref{upper est}.
Thus the exponential growth rate of $\| D b^n |_E(x) \|$ is at most
$\chi_M(b)$ for all $b$-invariant ergodic measures. Since
$\| D b^n |_E(x) \|$ is continuous, this implies the uniform exponential
growth estimate, as in the second inequality in \eqref{e-est}
(see \cite[Theorem 1]{Sch} or \cite[Proposition 3.4]{RH}).
The first inequality  in \eqref{e-est} follows similarly by
observing that the exponential growth rate of $\| D b^{-n} |_E(x) \|$
is at most $-\chi_m(b)$.
\end{proof}

\begin{Lemma}
Assume that there is an Anosov element in every Weyl chamber.
Then for any $a \in \Z^k$, $\a(a)$ is Anosov if and only if its
linearization $\rho(a)$ is Anosov.
\end{Lemma}

\begin{proof}
It is classical that if $a$ is Anosov so is it's linearization. So
assume that $\rho(a)$ is Anosov. Then $a$ does not belong
to any Lyapunov hyperplane of $\rho$ and hence of $\a$.
Then Proposition \ref{e-slow} applied to $a$ or $-a$ implies
that any coarse Lyapunov distribution of $\a$ is either
uniformly contracted or uniformly expanded by $\a(a)$.
This implies that $\a(a)$ is Anosov since the coarse Lyapunov
distributions span $TM$.\end{proof}


\subsection{Higher derivatives and estimates on compositions}
\label{subsection:jets}

In this subsection, we recall a basic estimate on higher derivatives
of compositions of diffeomorphisms.  The main point is that the
exponential growth rate is entirely controlled by the first
derivative.

Let $\psi$ be a diffeomorphism of a compact manifold $M$.  Given a function
$f$ on $M$, in local coordinates we have a vector valued function $f^k$
consisting of $f$ and it's partial derivatives up to order $k$.  Using a finite collection
of charts and a subordinate partition of unity, one can define the $C^k$ norm of $f$ as
$\sup_x\|f^k(x)\|$. It is easy to check that different choices of charts and/or partition of unity give
 rise to equivalent $C^k$ norms.  We will also write $\|f(x)\|_k=\|f^k(x)\|$ for the corresponding norm at $x$. More generally, let $\mathcal F$ be a foliation
of $M$ by smooth manifolds.  Given a function $f$ which is continuous and differentiable along $\mathcal F$ we can again
locally define a vector valued function $f^{k,\mathcal F}(x)$ consisting of $f$ and it's partial derivative
to order $k$ along $\mathcal F$ and let $\|f(x)\|_{k, \mathcal F}=\|f^{k,\mathcal F}(x)\|$.  Fixing a finite
collection of foliation charts and a subordinate partition of unity, this allows us to define $C^k$ norms corresponding to only taking derivatives along $\mathcal F$, by $\|f\|_{k,\mathcal F} = \sup _{x \in M} \|f(x)\|_{k,\mathcal F}$. Once again it is easy to check that different choices of charts and/or partition of unity give rise to equivalent norms. In this setting, for a homeomorphism $\psi$ of $M$ that is smooth along $\mathcal F$
with all derivatives continuous transversely, we define
$\| \psi(x)\|_{k,\mathcal F}=\sup \| f \circ \psi(x)\|_{k, {\mathcal F}}$ where the supremum is over functions
$f$ such that $\|f(\psi(x))\|_{k,\mathcal F}=1$. We then define
$\| \psi\|_{k,\mathcal F}=\sup_{x \in M} \|\psi(x)\|_{k, {\mathcal F}}$.

\begin{lemma}
\label{lemma:compositions} Let $\psi$ be a diffeomorphism of a
manifold $M$ preserving a foliation $\mathcal F$ by smooth leaves.
Let $N_k=\|\psi\|_{k,{\mathcal F}}$.  Then there exists a polynomial $P$ depending only on $k$ and the dimension of the leaves of $\mathcal F$ such that for every $m \in \mathbb N$
\begin{equation}
\label{equation:jetssup}
\|\psi^m\|_{k,{\mathcal F}} \leq N_1^{mk}P(mN_k).
\end{equation}
\end{lemma}

This type of estimate is used frequently in the dynamics literature particularly
in KAM theory and is usually referred to as an {\em estimate on compositions}.
This lemma is essentially \cite[Lemma 6.4]{FM} and a proof is contained
in Appendix $B$ of that paper.  There are many other proofs of
equation \ref{equation:jetssup} in the literature, though mostly only in the case where the foliation $\mathcal F$ is trivial, i.e. when the only leaf of $\mathcal F$ is the manifold $M$.  Most proofs should adapt easily to the foliated
setting.


\section{Exponential mixing for $\Z^{k}$-actions on tori}  \label{exponential}

Consider a  diffeomorphism  $a$ on a manifold preserving a probability measure $\mu$.   Given two H\"{o}lder
functions $f,g$, we consider the matrix coefficients $\langle a^{k} f, g\rangle$ where the bracket refers to the standard inner product on $L^{2}(\mu)$.
  For an Anosov diffeomorphism $a$,  the matrix coefficients of H\"{o}lder functions decay exponentially fast in $k$  for either an invariant volume or the measure of maximal entropy, as follows easily from symbolic dynamics.  D. Lind  established exponential decay for   H\"{o}lder functions for  ergodic toral automorphisms in \cite{Lind}. This is considerably harder, as there is no suitable symbolic dynamics.  Instead he shows that dual orbits of Fourier coefficients diverge fast as   one has good lower bounds on the distances from  integer points to neutral subspaces along stable and unstable subspaces.   This precisely is Katznelson's lemma  on rational approximation of invariant subspaces.  We adapt Lind's argument to prove exponential decay of matrix coefficients of H\"{o}lder functions for $\Z^{k}$ actions by ergodic automorphisms with a bound depending on the norm of the element in $\Z^{k}$.  Even if the $\Z ^{k}$-action contains only Anosov elements this is not trivial since we seek a bound in terms of the norm of $a \in \Z^{k}$.  In addition, some elements in $\Z^{k}$ will be arbitrarily close to the Lyapunov hyperplanes and thus have little, if any, expansion in certain directions.  Thus one  essentially  has to deal with the partially hyperbolic case.
 We remark that Damjanovi\'{c} and Katok obtained  estimates of exponential divergence  of Fourier coefficients  for the dual action induced by a  $\Z^{k}$-action by ergodic toral automorphisms   \cite{DamKat}.

 Finally, Gorodnik and the third author generalized exponential decay of matrix coefficients to ergodic automorphisms and $\Z ^{k}$ actions of such on  nilmanifolds \cite{GorodnikSpatzier}.  We will report on this development in  more detail in Section \ref{nilmanifolds} when we prove the nilmanifold version of our main result.  The arguments required for the   nilmanifold case  are substantially more complicated, and rely on work by Green and Tao on equidistribution of polynomial sequences \cite{Green-Tao}.   For this reason, and to keep our exposition for the case of toral automorphisms self contained and elementary, we present our adaptation of Lind's arguments.

Let $\tau$ be a $\Z ^k$ -action by ergodic automorphisms of $\T^n$.
We begin  by recalling  Katznelson's Lemma.  For a proof see
\cite[Lemma 4.1]{DamKat}.

\begin{lemma} Let $A$ be an $N\times N$ matrix with integer
coefficients. Suppose that $\R ^{N}$ splits as $\R ^{N}= V \oplus V' $ with  $V$ and $ V'$ invariant
under $A$ and such that $A\mid_{V}$ and $A\mid_{V'}$  do not have  common eigenvalues. If $V \cap \Z ^{N}=\{0\}$,
then there exists a constant $C$ such that
$$d(z,V)\geq C \| z \| ^{-N}$$
 for all $z \in \Z ^{N} $. Here $\| z \|$  denotes the  Euclidean norm and d the Euclidean distance.
\end{lemma}

Consider the finest decomposition into $\tau (\Z ^k)$-invariant subspaces $E_i$ of $\R ^n = \oplus_i  E_i $.  All $E_i$ are subspaces of
generalized eigenspaces of the elements of $\tau (\Z ^k)$. Let $\lambda _i$ denote the Lyapunov exponent defined by the vectors in $E_i$.  Then
$e ^{\lambda _i (a)}$ is the absolute value of the eigenvalue of $\tau (a)$ on $E_i$.   It is well-known that the $\lambda _i (a)$ are the Lyapunov exponents of $\tau (a)$.  Pick an inner product with respect to which   the $E_{j}$ are mutually orthogonal.  Let  $|\|v|\|$ denote its norm.  Since all norms on $\R^{n}$ are equivalent, we can pick  $D >0$ such that $\frac{1}{D} \| v\|  \leq |\|v|\|  \leq D \|v\|.$	Finally note that for any $a \in \Z^{k}$,  $\tau (a) $ expands $v \in E_{j}$ by at least    $e^{\lambda _{j}}(a)$.

\begin{lemma} If $\tau (a)$ is an ergodic toral automorphism, then for some $i$, $\lambda _i (a) \neq 0$.
\end{lemma}   \label{4.2}

{\colb This follows immediately from Kronecker's theorem that  the eigenvalues of an integer matrix are roots of unity if they all lie lie on the unit circle.  However, let us give a   simple direct proof. }

\begin{proof}  Consider the Jordan decomposition $\tau (a) = bc$ of $a$ where $b$ is semisimple, $c$ unipotent and $\tau (a)$ and $b$ commute.
Then for all $i$, $\lambda _i (a) = \lambda _i (b)$.  If all $\lambda _i (a) = 0$, then $b$ lies in a compact subgroup.  Since $\tau (a)$ is ergodic, no eigenvalue of $\tau (a)$ is a root of unity, and hence no power of $b$ is $1$.  Hence powers of $b$ approximate 1 arbitrarily closely.  Hence $tr \: \tau (a) ^l = tr \: b^l $ is arbitrarily close to $n$ for suitable $l$.  Since $\tau(a)^l \in SL(n,\Z)$,
$tr \: \tau (a)^l$ is an integer, and thus $tr \: \tau (a)^l =n$.  On the other hand, however,  $tr \: \tau (a)^l <n$ since the eigenvalues of $b^l$ cannot be real.   This is the final contradiction. \end{proof}

We will need a slightly stronger variant of this lemma. For $a \in \Z^k$, set $S(a) = \max _i  \lambda _i (\tau(a)) $.  Then $S(a) \neq 0$ for $\tau (a
)$ ergodic.

\begin{lemma} \label{maxexp}
Suppose all $0 \neq a \in \Z^k$, $\tau (a)$ acts ergodically.   Then $\inf \{S(a)\mid 0 \neq a \in \Z^k\} >0$.
\end{lemma}

{\colb Explicit lower bounds can be found in the literature, e.g. in \cite{BlanksbyMontgomery}.  We give an easy soft argument for a  positive lower bound.}

\begin{proof}   First suppose that all elements in $\Z ^k$ are semisimple.
If $\tau (a)$ is semisimple, then $\tau (a) $ expands each $E_{j}$ precisely by $e^{\lambda _j (a)}$ with respect to $|\|v|\|$.
Suppose $S(a_l) \rightarrow 0$ for a sequence of mutually distinct $1 \neq a_l \in \Z^k$.  Then  there are infinitely many $\tau (a_l)$ which expand distances w.r.t. $|\|\ldots|\|$ by at most $\frac{2}{D}$. Hence distances w.r.t. $\| \ldots \|$ get expanded by at most 2.
 Pick any integer vector $z \in \Z ^n$.   As  the images $a_l ( e_1)$ are integer vectors of norm at most $2 \: \| z \|$, for some $a_l \neq a_j$, $a_l ( z) =a_j ( z)$.  Hence $a_j ^{-1} a_l$ cannot be ergodic.

Next consider the general case.  Consider a generating set $a_1, \ldots, a_k$ of $\Z^k$.  Suppose $a_1 \in \Z^k$ has a Jordan
decomposition $\tau (a_1)=b_1 \: c_1$ with $b_1$ semisimple and $c_1$ unipotent.  Since $\tau (a_1) \in SL(n,\Z)$  both $b_1$ and $c_1$
are in $SL(n,\Q)$.  Since $c_1$ is unipotent, the subspace $W_1=\{ v \mid c_1 v =v \}$ of eigenvectors with eigenvalue 1 is
nontrivial and is defined over $\Q$.  Also, $W_1$ is $\tau (\Z^k)$-invariant, and $\tau (\Z^k)$ acts faithfully on $W_1$ since otherwise
some element $\tau (a) $ for  $a  \in \Z^k$ has eigenvalue 1 and is not ergodic.  Also $\tau (a)\mid _ {W_1}$ is semisimple.  Inductively, we
define a descending sequence of rational  $\tau (\Z^k)$-invariant subspaces $W_1 \supset W_2 \supset \ldots W_k$ on
 which $\Z^k$ acts faithfully. In addition, $\tau (a_i) \mid _{W_i} $ is semisimple.  Hence $\Z^k$ acts faithfully on $W_k$ and every element acts semisimply.  By the special case above, $\inf \{S(a\mid _{W_K})\mid 1 \neq a \in \Z^k\} >0$.  Since  $\inf \{S(a)\mid 0 \neq a \in \Z^k\}  \geq \inf \{S(a\mid _{W_K})\mid 0 \neq a \in \Z^k\}$, the claim follows.
\end{proof}

Note that the $\lambda _i$ and hence $S$ extend to continuous functions on $\R ^k$.

 \begin{lemma}  \label{4.4} 
Suppose for all  $0 \neq a \in \Z^k$, $\tau (a)$ acts ergodically.   Then for all $0 \neq a \in \R^k$, $S(a)>0$.  Thus $0 < \sigma := \frac12 \inf \{S(a) \mid a \in \R^k, \parallel a \parallel =1\}$.
\end{lemma}

\begin{proof}  Suppose $S(a)=0$ for some $0 \neq a \in \R^k$.   Since the line $ta, t \in \R$ comes arbitrarily close to integer points in $\Z^k$, we can find $t_l \in \R$ and $a_l \in \Z ^k$ with $a_l - t_l a \rightarrow 0$ as $l \rightarrow \infty$. As $S(t_l a)=0$ for all $l$, it follows readily that $S(a_l) \rightarrow 0$ in contradiction to the last lemma.  The last claim follows as $S$ is continuous.
\end{proof}

Let $B(d) $ denote the ball of radius $d$  in $\Z ^k$.

\begin{lemma}   \label{4.5} 
Let $1 < r < e^{\frac{\sigma}{n+2}}$.  Set $H_l = \{ z \in \Z \mid - r^l \leq z \leq r^l\} ^n$. Then we have   for all sufficiently large $l$ and $a \in \Z^k$ with $\|a\| \geq l$
\[  \tau (a) (H_l) \cap  H_l = \{0\}.\]
\end{lemma}

\begin{proof}  Fix a  constant $b >0$ such that for all $r>0$, $ [-r,r] ^n $ is contained in the ball $  B_{b r} (0)$ of radius $b r$ about 0.

Suppose that there is a sequence $l_m \rightarrow \infty$ and $a_{l_m} \in \Z^k$ with $\alpha _{l_{m}} := \| a_{l_m}  \| \geq l_m$  such that  $ \tau (a_{l_m})(H_{l_m}) \cap  H_{l_m} \neq \{0\}$.    Passing to a subsequence we may assume that
$\frac{a_{l_m}}{\alpha_{l_m}} \rightarrow a$ converges to $a \in \R^k$.   Since $S(a) \geq 2 \sigma$, $\lambda _i (a) \geq \sigma$ for some $i$.  Hence we get for all large $m$ that $\lambda _i (a_{l_m}) \geq  l_m \sigma $.

 Let $E= \oplus _{j \neq i} E_j$.  By Katznelson's Lemma applied to $E$, there is a constant $C >0$ such that for $0 \neq z \in \Z^n$, the distance $d(z, E) > C \| z \| ^{-n}$.
Suppose $z_{l_m} \in H_{l_m}$ with $\tau (a_{l_m}) z _{l_m}\in H_{l_m}$. Then we  get
$$\|  z_{l_m} \| < b r^{l_m} \text{ and  }\| \tau (a_{l_m}) z_{l_m} \| < b r^{l_m}.  $$

Denote by $\pi _i$ the projection to $E_i$ along $E$. Then $\| \pi_i (z_{l_m}) \| = d(z_{l_m},E) \geq C \|z_{l_m}\| ^{-n} > C b^{-n} r ^{-n l_m}.$

 As $E$ and $E_i$ are transversal and have constant angle, there is a constant $M$ such that for all $v \in \R^n$, $\pi _i (v) \leq M \|v\|$.
Hence $\| \tau (a _{l_m} )(\pi _i (z_{l_m}))\| = \|\pi _i (\tau (a_{l_m}) z_{l_m}) \| < M b r^{l_m}$.   On the other hand,  we will show below that
$$ \|\tau (a_{l_m}) (\pi _i (z_{l_m})) \| \geq {  \frac{1}{D} } e^{\sigma l_m} b^{-n} r^{-nl_{m}}.$$
 Indeed, this estimate is clear when $\tau (a)$ is semisimple but needs more care when $\tau(a)$ has nontrivial Jordan form.
   This estimate will yield  a contradiction to the Lyapunov exponent $\lambda _i (a)$ of $a$ to be at least $\sigma$.  Here is the detail.

Set $v_{l_m} := \frac{\pi _i (z_{l_m})}{\|\pi _i (z_{l_m})\|}$. By the estimates above we get
\[ \|\tau ( a_{l_m} )(v_{l_m}) \|\leq \frac{M b r^{l_m} }{\|   \pi_i z_{l_m} \| } \leq  M C^{-1} b^{n+1} r^{(n+1) l_m}.\]

Set  $b _{l_m} := a -  \frac{a_{l_m}}{\alpha _{l_m}}$.   Then $b_{l_m} \rightarrow 0$.  For all large $m$, we may assume that $b_{l_m}$ expands vectors by a factor of at most $r$.   Since $l_m \le \alpha_{l_m}$ this  implies

\[ \|\tau (\alpha _{l_m} a) (v_{l_m}) \| = \| \tau (\alpha _{l_m} b _{l_m}) \tau ( a _{l_m}) (v_{l_m}) \| \leq MC^{-1} b^{n+1} r^{(n+1) l_m} r^{ \alpha _{l_m}}\leq  M C^{-1} b^{n+1} r^{(n+2)\alpha_{ l_m}}\]

Find a basis $w_1, \ldots w_s$ of $E_i$ which brings $a$ to Jordan form.  Write $v_{l_m} = x^1 _{l_m} w_1 + \ldots +x^s _{l_m} w_s$.
Passing to a subsequence the $v_{l_m}$ converge.  Suppose $j$ is the last  coordinate  such that $x^j _{l_m} \rightarrow x^j \neq 0$.  Then  $\tau (\alpha _{l_m} a) (v_{l_m}) $ has $j$-coordinate of absolute value $ x^{j} e^{\alpha _{l_m} \: \lambda _i (a)} $.    Since the sup  norm determined by the basis $w_1, \ldots, w_r$ is equivalent to the standard Euclidean norm,  there is a constant M' such that $ \|\tau(\alpha _{l_m}a) (v_{l_m}) \|  > M' x^{j} e^{\alpha _{l_m} \: \lambda _i (a)} $.  Hence
\[ M' x^{j} e^{\alpha _{l_m} \: \sigma } < M' x^{j} e^{\alpha _{l_m} \: \lambda _i (a)}    <        M C^{-1} b^{n+1} r^{(n+2)\alpha_{ l_m}}.\]
This is impossible for large $l_m$ by choice of $r$ and $\sigma$.
\end{proof}

We will use the approximation by Fej\'{e}r kernel functions
 $K_l (t) = \sum _{j=-l} ^l \big( 1- \frac{|j|}{l+1}\big) e ^{2 \pi i j t} $, and refer to \cite[chapter I]{Katznelson} for  details.

 Set $F_l (t_1, \ldots, t_n) = K_l (t_1) \ldots  K_l (t_n)$.  For  continuous $f: {\T} ^n \mapsto \R$, $ K_l \star f$ is supported on $H_l$.
Endow  the space
\[H_{\theta} = \{ f: \T^n \rightarrow \R \mid f \text{ is H\"{o}lder with H\"{o}lder exponent } \theta\ \} \]
for $0 < \theta <1$
 with the norm
\[ \| f \| _{\theta} = \|f \| _{\infty} + \sup _{t, h \neq 0} \frac{\| f( t+ h) - f(t) \|}{ \|h\| ^{\theta} }  .\]
As in \cite[p. 21, Exercise 1]{Katznelson}, we get

\begin{lemma}
There is a constant $C = C ( \theta)$ such that the map $H_{\theta} \mapsto L_{\infty} (\T^n)$ given by $ f \mapsto F_m \star f$ satisfies the estimate
\[ \| F_m \star f -f \|_{\infty}   \leq C(\theta) \|f\| _{\theta} m ^{ - \theta} .\]

 \end{lemma}

\begin{Theorem}\label{matrixcoef}
Suppose $\Z^{k}$ acts affinely on $\T^{n}$ such that  for all  $0 \neq a \in \Z^k$, $\tau (a)$ acts ergodically.  Let $f$ and $g$ be two H\"{o}lder functions on
$\T^n$ with 
H\"{o}lder exponents $\theta$. {\colb Then there exists $r>1$ such that}  for any $a_l \in \Z^k$ with $\| a_l \|
\geq l $ we can bound the matrix coefficients
\[\left|\langle a_l f, g\rangle  - \int _{\T^n} f  \int  _{\T^n} g \right|  <  C(\theta)  \:\:  \left( 4   \|f\| _{\theta} \|g\|_2 + 2   \|g\| _{\theta} \|f\|_2 \right)\:\:    r^{-\theta l}   \]
 In particular, the matrix coefficients decay exponentially fast.
\end{Theorem}

\begin{proof} We can can assume that  $\int _{\T^n} f = \int  _{\T^n} g =0$ are both 0 by subtracting the constants $\int _{\T^n} f$ and $ \int  _{\T^n} g$ from $f$ and $g$ respectively.

{\colb We  pick $1 < r < e^{\frac{\sigma}{n+2}} $ as in Lemma \ref{4.5} where $\sigma$ is as in Lemma  \ref{4.4}.}
 Let $m= [r^l] $ , the largest integer smaller than $r^l$.  Set  $f_l =K_ m \star f$ and $g_l = K_m \star g$ with frequencies in $H_l$.  Then  $\int _{\T^n} f_l = \int  _{\T^n} g_l =0$ and
$\| f - f_l \| _{\infty} \leq  2 C(\theta)  \|f\| _{\theta} (r^l) ^{-\theta} $ and  $\| g - g_l \| _{\infty} < 2 C(\theta) \|g\|_{\theta}  (r^l) ^{-\theta} $ where the 2 accounts for the discrepancy coming from $m$ versus $r^l$.     By the last lemma, we get
\[ \langle a_l (f), g\rangle = \langle a_l f, (g- g_l)\rangle + \langle a_l (f-f_l), g_l \rangle + \langle a_l (f_l), g_l\rangle .\]
The last term is eventually 0 since the constant term is 0 and $a_l$ moves $H_l$ off itself.  The first  term is bounded by
\[ \|f\|_2 \|g-g_l\|_{\infty} \leq 2  C(\theta) \|g\| _{\theta} \|f\|_2  \:\: r^{-\theta l}
.  \]
 Take $l$ large enough so that $\| g - g_l \| _{\infty} < 2 C(\theta) \|g\|_{\theta}  (r^l) ^{-\theta} <2$,  Then the second term is bounded by
\[ \|g_l\|_2 \|f-f_l\|_{\infty} \leq    2    C(\theta)  \|f\|_{\theta} \|g_l\|_2  \:\: r^{-\theta l} \leq    4    C(\theta)  \|f\|_{\theta} \|g\|_2  \:\: r^{-\theta l} .\]
   This yields the desired estimate   \end{proof}

\begin{corollary}
\label{corollary:matrixcoef} The same statement as above holds for
any Anosov $\Z^k$ action with $k>1$  where every element acts
ergodically.
\end{corollary}

\begin{proof}
This combines Theorem \ref{matrixcoef}, the existence of a H\"{o}lder conjugacy,
and the fact that we define matrix coefficients with respect to the pushforward
measure which is the unique smooth invariant measure by Proposition \ref{mu}.
\end{proof}


{\colb
\section{ Regularity and the Proof of Theorem 1.1}    \label{soln-cohomology}
}

In this section we complete the proof of Theorem \ref{theorem:main}
by showing that the Franks-Manning conjugacy $\phi$ between
the $\Z ^k$-actions $\a$ and $\rho$ is smooth.  We will use $\phi$
and the uniform exponential estimates along the coarse Lyapunov
foliations of $\a$ from Section \ref{uniform estimates}, but we will not
use Anosov elements explicitly in this section. Instead, we will use
the subgroup $\Z ^{2}$ consisting of ergodic elements that we postulated
in Theorem \ref{theorem:main}.  Theorem \ref{matrixcoef} gives exponential
mixing with uniform estimates along this $\Z ^{2}$. This allows us to define
distributions on H\"{o}lder functions which correspond to the components
of the conjugacy and their derivatives. First however, we  will make some reductions to the general case.

By passing to a finite index subgroup of $\Z^k$
we can assume that the action $\a$ has a common fixed point.
First we reduce the problem to the case when $\a$ acts on the torus with
the standard differentiable structure. Note that a construction due to
Farrell and Jones shows that there exist Anosov diffeomorphisms
of exotic tori \cite{FarrellJones}.  However, every exotic torus of dimension
at least 5 has a finite cover which is diffeomorphic to the standard torus
\cite[Chapter 15 A, last unitalized paragraph]{Wall}. In this case we can
consider the lifts of the actions and the conjugacy. Clearly, the smoothness
of $\phi$ follows from the smoothness of its lift. We will give an independent
 argument in Section \ref{four} for the case of 4-dimensional tori. Hence, without
loss of generality, we can assume that $\a$ acts on the same standard torus
as $\rho$. {\colb In dimensions $2$ and $3$, by Remark \ref{non-exotic} in the Appendix,  there are no exotic
differentiable structures, though this fact is not strictly needed here.  In dimension $3$, Theorem \ref{theorem:main}
follows from the main result of \cite{RH}.  As explained in Section \ref{four}, 
there are no higher rank Anosov actions on tori in dimension $2$.}

By changing coordinates we can also assume that $0$ is a common fixed
point for both $\a$ and $\rho$. Then there exists a unique conjugacy $\phi$
in the homotopy class of identity satisfying $\phi (0)=0$. We can lift  $\phi$ to
the map $\tilde \phi : \R^n \to \R^n$ satisfying $\tilde \phi (0)=0$ and write it
as $\tilde \phi=I + h$, where $h: \R^n \to \R^n$ is $\Z^n$ periodic.

 Consider  an
element $a$ in $\Z ^2$  and abbreviate $\a(a)$ to $a$ and
$\rho (a)$ to $\bar A$. We denote their lifts to $\R^n$ that fix $0$ by
$\tilde a$ and $A$ respectively and note that $A$ is linear.  Since $\phi$
is a conjugacy and the lifts fix $0$, they satisfy $\tilde \phi \circ\tilde a = A \circ \tilde \phi$. Hence we obtain
  \[ (I+h) (\tilde a (x)) = A \,(I+h) (x),\]
 which is equivalent to
\[
 h(x) = A^{-1} (\tilde a(x) -A(x)) + A ^{-1} (h(\tilde a(x))) = Q(x) + A ^{-1} (h(\tilde a(x)))
\]
where $Q(x) = A^{- 1} (\tilde a(x) - A(x))$. Note that  $Q(x)$ is smooth since
 $a$ is smooth with respect to the standard differentiable structure
(this will be crucial later). Since $h$ is $\Z^n$ periodic it is easy to see
that $A ^{-1} (h(\tilde a(x)))$ and hence $Q(x)$ are also $\Z^n$ periodic.
For the remainder of this section we will view $h$ and $Q$ as functions
from $\T^n$ to $\R^n$. The functional equation on $\T^n$ becomes
\begin{equation} \label{hoper}
 h(x) = Q(x) + A ^{-1} (h(ax)).
 \end{equation}

Fix a coarse Lyapunov foliation $\mathcal V$ of $\a$ and the corresponding
linear coarse Lyapunov foliation $\bar {\mathcal V}$ of $\rho$.
 Let $V$ be the subspace of $\R^n$ parallel
to $\bar {\mathcal V}$ and $W$ be the complementary $A$ invariant
subspace, which is parallel to the sum of all coarse Lyapunov foliations
of $\rho$ different from $\bar {\mathcal V}$. Denote by $h_V: \R^n \to V$
the projection of $h$ to $V$  along $W$. Since  $V$ is $A$-invariant, projecting
equation $(\ref{hoper})$ and letting $A_V$ denote the restriction of $A$ to $V$
we obtain
\begin{equation} \label{hVoper}
 h_V (x) = Q_V(x) + A_V ^{-1} (h_V(ax)) =: F_V(h_V) (x)
 \end{equation}
where $Q_V$ denotes the projection of $Q$ to $V$ along $W$.

We will use the functional equation \eqref{hVoper} with well-chosen elements
 $a$ to study the derivatives
of $h_V$ along the coarse Lyapunov foliations of $\a$.  These derivatives
exist, a priori, only in the sense of distribution on smooth functions. The
crucial element of the proof is Lemma \ref{lemma:derivativesofh} below
which shows that these distributional derivatives extend to functionals on the
spaces of H\"{o}lder functions.
{\colb
We emphasize that this lemma is quite general, and may be useful in other situations.  The main ingredients are the uniform exponential estimates with arbitrarily small exponents along coarse Lyapunov foliations,  and exponential mixing for H\"older functions.  The key idea is that in our estimates for derivatives,  the exponential decay coming from exponential mixing overcomes  small exponential growth coming from derivatives.
}

\begin{lemma}
\label{lemma:derivativesofh} For any coarse Lyapunov foliation
$\mathcal V'$ of $\a$, possibly equal to $\mathcal V$, and for any
$\theta>0$ the derivatives of $h_V$ of any order along $\mathcal V'$ exist as distributions on the space of $\theta$-H\"{o}lder functions.
\end{lemma}

\begin{proof}
Let $L,L^+,L^- \subset \R^k$ be the Lyapunov hyperplane and the positive
and negative Lyapunov half-spaces corresponding to $\mathcal V$.
Let $L'$ be the Lyapunov hyperplane corresponding to $\mathcal V'$.
In this proof we will choose $a$ in the $\Z^2$ subgroup consisting of
ergodic elements. We note that $\mathcal V$ and $\mathcal V'$ are
coarse Lyapunov foliations for $\a$-action of the full $\Z^k$ and that we make
no assumptions on the relative positions of  $\Z^2$, $L$, and $L'$ in
$\R^k$. We will choose $a$ in a
narrow cone in $\Z^2$ around $L' \cap \Z^2$, so that $a$ will  expand
$\mathcal V'$ at most slowly. In case $\Z^2 \subset L'$, this automatically
holds for all $a$ in $\Z^2$. Since any such cone can not be contained entirely in
$L^-$, we can always choose such an $a\in \Z^2$ in $L^+$ or $L$.

If $a\in L^+$ then $A_V^{-1}$ is a contraction. Then the operator $F_V$
in \eqref{hVoper} is a contraction on the space $C^0 (\T^n,V)$. Hence it
has a unique fixed point $\lim F_V^m(0)$, which therefore has to coincide
with $h_V$. Thus we obtain
\begin{equation} \label{hVseries}
h_V (x) = \sum _{m=0} ^{\infty} A_V^{-m }\:Q_V (a^{m} x).
\end{equation}

If $a\in L$ the series in \eqref{hVseries} does not converge in the
space of continuous functions. However, it converges in the space
$\mathcal D_0$ of distributions on smooth functions with zero average,
and the equality in \eqref{hVseries} holds in $\mathcal D_0$.
To see this we iterate \eqref{hVoper} to get
\begin{equation} \label{hVsum}
h_V (x) = \sum _{m=0} ^{N-1} A_V^{-m }\:Q_V (a^{m} x)
 +  A_V^{-N} h_V (a^{N} x).
\end{equation}
Since $\| A_V^{-m}\|$ grows at most polynomially in $m$ for $a\in L$,
and since $h_V$ is H\"older, Corollary \ref{corollary:matrixcoef} implies
that the pairing $\langle A_V ^{-N} h_V ( a^N x) , f \rangle   \to 0$ for
any H\"older function $f$ with $\int _{\T^n} f =0$. This establishes convergence and equality in \eqref{hVseries}
when both sides are considered as elements in $\mathcal D_0$.

We will use notations of Section \ref{subsection:jets} for derivatives. Given a
smooth function $g:\T^n\rightarrow \R^l$, we write $g^{k, \mathcal V'}$ for the
vector consisting of the derivatives of $g$ up to order $k$ along the foliation
$\mathcal V'$. If $g$ is a vector valued function on $\T^n$ and $f$ is a scalar valued function, we write $gf$ for the vector function obtained by
component-wise multiplication of $g$ by $f$. We then write
$\langle g,f \rangle $
for the vector obtained by integrating $gf$ over $\T^n$. We will use the
same notation $h^{k,{\mathcal V'}}_V$ for the vector of distributional
derivatives of $h_V$ along $\mathcal V'$ (see Section \ref{fromrauch}
for detailed description of distributional derivatives in the context of foliations).
Differentiating \eqref{hVseries} term-wise we obtain the formula for
$h^{k,{\mathcal V'}}_V$
\begin{equation} \label{h_V^k}
\langle h^{k,{\mathcal V'}}_V,f \rangle=
\sum_{m=0}^\infty \langle  A_V^{-m}(Q_V\circ a^m)^{k, \mathcal V'},f\rangle.
\end{equation}
Note that the derivative of a distribution is defined by its values on
derivatives of test functions \eqref{equation:der}, and those have
zero average. Thus  convergence and equality in \eqref{h_V^k}
hold in the space $\mathcal D$ of distributions on smooth functions,
even if equality in \eqref{hVseries} hold only in $\mathcal D_0$.
Since $Q_V$ is smooth, the pairings in the series in \eqref{h_V^k}
are simply given by integration. To show that $h^{k,{\mathcal V'}}_V$
extends to a functional on the space of H\"{o}lder functions we will
now estimate these pairings in terms of the H\"{o}lder norm of $f$.

We will use smooth approximations of $f$ by convolutions
$f_{\varepsilon} = f \star \phi _{\varepsilon}$, where the kernel is
given by rescaling
$\phi _{\varepsilon} (x)= \varepsilon ^{-n}  \phi (\frac x{\varepsilon})$
of a fixed bump function $\phi$ and thus is supported on the ball of
radius $\varepsilon$ and satisfies
$$
\phi _{\varepsilon} \ge 0, \quad \int _{\T^n} \phi _{\varepsilon} =1, \quad
\|  \phi _{\varepsilon} \| _{C^k}= \varepsilon ^{-(n+k)}  \|\phi \| _{C^k}.
$$

Then it is  easy to check the following estimates, where $\| .\|_k$ denotes the $C^k$ norm for $k\ge0$,
 \begin{equation} \label{fe}
 \|f_{\varepsilon } -f \|_{0} \leq \varepsilon ^{\theta} \|f\|_{\theta} \quad
 \text{ for } 0<\theta \le 1
 \quad \text{and} \qquad
\|f_{\varepsilon}\|_{C^k}  \leq c_k \, \varepsilon ^{-n-k}  \|f\| _{0} \quad
 \text{ for } k \in \mathbb N
\end{equation}
\noindent where $f$ is a $\theta$-H\"{o}lder function and $c_k$ is a constant
depending only on $k$.
 First we estimate the pairings in \eqref{h_V^k} with $f_\e$. Note that $\| .\|_l \le \| .\|_k$ if $l\le k$.
 We have

\begin{multline*}
\|\langle A_V^{-m} (Q_V \circ a^m)^{k,{\mathcal V'}},f_\e\rangle \|  \le
\|A_V^{-m}\| \cdot \|\langle (Q_V \circ a^m)^{k,{\mathcal V'}},f_\e\rangle \| =\\
= \|A_V^{-m}\| \cdot  \|\langle Q_V \circ a^m,(f_\e)^{k,{\mathcal V'}}\rangle \|\, .
\end{multline*}
Since $\| (f_\e)^{k,{\mathcal V'}} \|_\theta \le \| (f_\e)^{k,{\mathcal V'}} \|_{1}
\le \| f_\e \|_{{k+1}}$,
using Corollary \ref{corollary:matrixcoef} and \eqref{fe} we can estimate
\begin{multline*}
\|\langle Q_V \circ a^m,(f_\e)^{k,{\mathcal V'}}\rangle \|  \le
 K_1 \,  r^{- m \|a\| \theta} \, \|Q_V \|_\theta \| (f_\e)^{k,{\mathcal V'}} \|_\theta
  \le\\ \le K_2 \, r^{- m \|a\| \theta}  \varepsilon ^{-(n+k+1)}
  \|Q_V \|_\theta \, \|f\| _{0} \, .
\end{multline*}
Since  $a$ is chosen in  $L^+ \cup L$, $\| A_V^{-1}\|$ grows at most polynomially in $\| a\|$ and thus, for any $\eta>0$, we can ensure
that $\| A_V^{-1}\|< (1+\eta)^{\|a\|}$ for all $a$ with sufficiently large norm.
Thus we conclude from the two equations above that
\begin{equation} \label{pfe}
\|\langle A_V^{-m} (Q_V \circ a^m)^{k,{\mathcal V'}},f_\e\rangle \|  \le
K_2 \, (1+\eta)^{m\|a\|} \, r^{- m \|a\| \theta} \, \varepsilon ^{-(n+k+1)}
\|Q_V \|_\theta \, \|f\| _{0} \, .
\end{equation}

Now we estimate the pairings in \eqref{h_V^k} with $f-f_\e$ using the supremum norm and estimating $\| A_V^{-m}\|$ as above
  $$
\|\langle A_V^{-m} (Q_V \circ a^m)^{k,{\mathcal V'}},(f-f_\e)\rangle \|  \le
\|A_V^{-m} (Q_V \circ a^m)^{k,{\mathcal V'}}\|_0 \cdot \| (f-f_\e)\|_0 \le
$$
$$
\|A_V^{-m}\| \cdot \| Q_V \circ a^m \|_{k,\mathcal V'} \cdot \e ^{\theta} \|f\|_{\theta} \le
(1+\eta)^{m\|a\|} \cdot \| a^m \|_{k,\mathcal V'} \cdot \|  Q_V \|_{k,\mathcal V'}  \cdot \e ^{\theta} \|f\|_{\theta}.
$$
Here we used notations of Section \ref{subsection:jets}. Denoting
$N_k = \|a\|_{k,\mathcal V'}$, and using equation \eqref{equation:jetssup}
from Lemma \ref{lemma:compositions} we conclude that
$$
\|\langle A_V^{-m} (Q_V \circ a^m)^{k,{\mathcal V'}},(f-f_\e)\rangle \|  \le
(1+\eta)^{m\|a\|} \cdot N_1^{mk} \, P(mN_k)\cdot \e ^{\theta} \cdot \|  Q_V \|_k  \cdot  \|f\|_{\theta}\, .
$$
Recall that we choose $a$ in a cone around $L' \cap \Z^2$. For any
$\eta>0$, by taking the cone sufficiently narrow and using Proposition
\ref{e-slow}, we can ensure that $N_1=\|a\|_{1, \mathcal V'} < (1+\eta)^{\|a\|}$ for any such $a$ with sufficiently large norm. Then from the last equation
we obtain that
\begin{equation} \label{pf-fe}
\|\langle A_V^{-m} (Q_V \circ a^m)^{k,{\mathcal V'}},(f-f_\e)\rangle \|
\leq  (1+\eta)^{m(k+1)\|a\|} \cdot P(mN_k) \cdot \e ^{\theta} \cdot \|  Q_V \|_k  \cdot  \|f\|_{\theta} \, .
\end{equation}

For any fixed $\theta$, we have a fixed rate of exponential
decay with respect to $m$ in \eqref{pfe}, but  the rate of
exponential growth in \eqref{pf-fe} can be made arbitrarily slow.
This allows us to choose $\e$ that gives exponentially decaying
estimates for both \eqref{pfe} and \eqref{pf-fe}. More precisely,
we take
$$
\varepsilon = r^{\frac{- m \|a\| \theta}{\theta + n +k+1}}
\qquad \text{and denote} \quad \zeta=r^{\frac{\theta^2}{\theta + n +k+1}} >1.
$$
Then we obtain from \eqref{pfe} and \eqref{pf-fe} that
$$
\|\langle A_V^{-m} (Q_V \circ a^m)^{k,{\mathcal V'}},f_\e\rangle \|  \le
K_2 \, (1+\eta)^{m\|a\|} \, \zeta^{- m \|a\|}  \cdot
 \|Q_V \|_\theta \, \|f\| _{0} \quad \text{and}
$$
$$
\|\langle A_V^{-m} (Q_V \circ a^m)^{k,{\mathcal V'}},(f-f_\e)\rangle \|
\leq  (1+\eta)^{(k+1)m\|a\|} \, P(mN_k) \; \zeta^{- m \|a\|} \cdot \|  Q_V \|_k \, \|f\|_{\theta}.
$$
For any $k$ we can now choose $a$, and hence $\eta$, so that
$\xi =\zeta \cdot (1+\eta)^{-(k+2)} >1$.  Since the polynomial $P$ and constant
$N_k$  depend only on $k$ and $a$, we can then estimate
$P(mN_k) \le K_3 (1+\eta)^{m\|a\|}$. Finally, we obtain from the last two equations that
$$
\|\langle A_V^{-m} (Q_V \circ a^m)^{k,{\mathcal V'}},f\rangle \|
\leq  K_4 \, \xi^{- m \|a\|} \cdot \|  Q_V \|_k \, \|f\|_{\theta}.
$$
Thus for any $\theta$ and $k$ we obtain
exponentially decreasing estimates for the terms in \eqref{h_V^k}. We conclude that $ \| \langle h^{k,{\mathcal V'}}_V,f \rangle \| \le C \|f\|_{\theta}$
and hence $h^{k,{\mathcal V'}}_V$ extends to a functional on the space
of $\theta$-H\"{o}lder functions.

{\colb
{\em Proof of Theorem 1.1:} We discussed actions on two- and three dimensional tori above, and will prove Theorem 1.1 for four-dimensional tori with an exotic smooth structure in the next section.    When the dimension is greater than four, as explained above, we can  pass to a finite cover by smoothing  theory  and assume that the smooth structure is standard.  Passing to a subgroup of finite index, we can also assume that $\alpha$ has a common fixed point.  By Lemma \ref{lemma:derivativesofh},  for any coarse Lyapunov foliation
$\mathcal V'$ of $\a$ and for any
$\theta>0$ the derivatives of $h_V$ of any order along $\mathcal V'$ exist as distributions on the space of $\theta$-H\"{o}lder functions.  Hence by Corollary~\ref{what we need}, all $h_V$ are $C^{\infty}$.}
{\colb Since
the subspaces $V$ span, $h$ is determined by the projections  $h_V$.
It follows that $h$ is $C^{\infty}$ and hence so is $\phi$. It remains to
show that $\phi$ is a diffeomorphism. Since $\phi$ is a homeomorphism,
it suffices to show that the differential of $\phi$ is everywhere non-degenerate. This  follows from Proposition \ref{mu} since we have
$\lambda = \phi _\ast (\mu)$ and $\mu$ has smooth positive density.

}

\end{proof}

\section{Four Dimensional Exotic Tori}   \label{four}

Now consider a higher rank Anosov action on a 4-dimensional
torus with an exotic differentiable structure.  Due to low dimension
we are able adapt arguments from \cite{FKS} to obtain the result in this case.

By passing to a finite index subgroup of $\Z^k$ we can assume that the
linear part $\rho$ acts by linear automorphisms from $SL(4,\Z)$.
We begin by analyzing possibilities for such actions on $\T^4$.
Let $A \in SL(4,\Z)$ be an Anosov element for $\rho$.  First we
claim that the characteristic polynomial of $A$ is irreducible over
$\Q$. Indeed, the only possible splitting would be into a
product of quadratic terms and would imply existence of a rational
invariant subspace of dimension two. Such a subspace would be
invariant with respect to a finite index subgroup of $\Z^k$. The
restriction of $\rho$ to the corresponding torus would still be Anosov
and contain a $\Z^2$ subgroup of ergodic elements, as ergodicity
is equivalent to having no root of unity as eigenvalue. The latter
however is impossible since Anosov actions on $\T^2$ can only have
rank one. More precisely, by the Dirichlet Unit Theorem the centralizer
of an irreducible Anosov matrix in $SL(n,\Z)$ is a finite extension
of $\Z^d$, where $d$ is $n-1$ minus the number of pairs of complex
eigenvalues. {\colb Moreover, all nontrivial elements of this $\Z^d$ are
semisimple. We conclude that $\rho(\Z^k)$ is a subgroup
of such $\Z^d \subset SL(4,\Z)$.}

We note that $\rho$ has four Lyapunov exponents (counted with multiplicity)
and $\chi_1+\chi_2+\chi_3+\chi_4=0$ by volume preservation. If no two are
negatively proportional then $\rho$, and hence $\a$, are so called TNS (totally nonsymplectic) and smoothness of the conjugacy follows from
 \cite[Theorem 1.1]{FKS}. Now suppose that there are negatively
proportional Lyapunov exponents. This case does not follow from any previous
theorem but can still be handled using techniques from \cite{FKS} and \cite{KaSa3}.
Note that in this case there are no positively proportional Lyapunov exponents,
as otherwise for elements near the
kernel of the negatively proportional ones all Lyapunov exponents
will be close to zero by volume preservation, contradicting Lemma \ref{maxexp}.
{\colb This implies that $\rho(\Z^k)$ contains matrices}
with pure real spectrum
and the coarse Lyapunov spaces for $\rho$ are one-dimensional and
totally irrational, so in particular the corresponding linear foliations of
$\T^4$ are ergodic.

For the nonlinear action $\a$ the coarse Lyapunov foliations are also
one-dimensional and any pair $W_i,W_j$ is jointly integrable in topological
sense by the conjugacy to the linear action. By \cite[Lemma 4.1]{KaSa3}
the joint foliation $W_{ij}$ has smooth leaves. {\colb For each $W_i$ consider a  $W_j$ which does not
correspond to negatively proportional exponents.   Then } one can see as in
 \cite[Proposition 5.2]{FKS} that there is an element that contracts $W_i$
faster than $W_j$ and conclude that  $W_i$ and $W_j$ are $C^\infty$
along the leaves of $W_{ij}$. In place of measurable normal forms in \cite{FKS},
for one-dimensional foliations we can use the nonstationary
linearization \cite[Proposition A.1]{KaLe} which is continuous on $M$ in the
$C^\infty$ topology. Hence a simple version of the holonomy argument
\cite[Proposition 8.1]{FKS} works for any $W_i$ using the holonomy
along such $W_j$. The argument shows that the conjugacy $\phi$
is $C^\infty$ along any $W_i(x)$ with the derivatives continuous on $M$.
Then the smoothness of $\phi$  follows  easily as in \cite{FKS}.



\section{The nilmanifold Case}
\label{nilmanifolds}

In this section we will describe the adaptations of our arguments   needed for the case of an Anosov action on an infranilmanifold  $M$.   Passing to finite covers, we can assume that  $ N/\Gamma$ is a nilmanifold.  Next we reduce to the case when the differentiable structure on  $ N/\Gamma$  is standard, i.e. given by the ambient Lie group structure. First we note that there are no
 nilmanifolds of dimension at most 4 supporting an Anosov automorphism  besides the torus.   Hence we can employ the theorem of J. Davis, proved in the appendix, that every exotic nilmanifold in dimension at least 5 has a finite cover with standard differentiable structure.  This allows to lift the actions to ones smooth with respect to a standard differentiable structure,  as in the beginning of Section \ref{soln-cohomology}.   Thus  the main theorem follows for nilmanifolds of dimension at least 5 provided it holds for actions on standard nilmanifolds.  We will now give a proof of the main theorem in this set-up .

First note that the arguments from Section 3 allowing uniform control of exponents work verbatim.   That certain distributions are dual to the space of H\"{o}lder functions  will again be key  to our arguments. This requires exponential mixing of the action which does not follow easily from  Fourier analysis or more generally representation theory anymore.  Instead we evoke a recent result by Gorodnik and the third author \cite{GorodnikSpatzierII}.  This is far less elementary than the results in Section \ref{exponential}, and use recent results of Green and Tao \cite{Green-Tao} on equidistribution of polynomial sequences.

\begin{Theorem} [Gorodnik-Spatzier]\label{exp-higher}
Consider a  $\Z^k$ action  $\alpha$  by ergodic affine diffeomorphisms  on an infra-nilmanifold.  Then for any $0< \theta <1$   there is $0< \lambda <1$ such that for any two $\theta$-H\"{o}lder functions $f,g: X \mapsto \R$ 
we get
\begin{align}
\left| \langle f \circ \alpha(z), g \rangle  -  \int _{\T^n} f  \int  _{\T^n} g   \right|    \leq O_{\theta}(\lambda ^{\|z\|} )  \|f\|_{\theta}  \|g\|_{\theta}
\end{align}
where $\|z\|$ denotes some fixed norm on $\Z^k$.
\end{Theorem}

We need to establish regularity of the solutions to   the cocycle equations employed in Section \ref{soln-cohomology}.  We are inspired by the approach of Margulis and Qian in \cite[Lemma 6.5]{MaQi}.   However, while they write their equations in exponential coordinates and directly  study the solutions in these coordinates, we will reduce the cocoycle equation to a series of equations, one for each term of the derived series of $N$.   This yields abelian valued cocycle equations to which we can apply the arguments from the toral case.    Here are the details.

As in Section \ref{soln-cohomology} we consider the lift $\tilde \phi : N \to N$
of the Franks-Manning conjugacy $\phi: N/\Gamma \to N/\Gamma$.
We can write it as a product $\tilde \phi = h  \cdot I$, where $h :N \to N$
satisfies
\begin{align}  \label{cocycleB}
 (h \cdot I) (a(x)) = A \,\big((h \cdot I) (x)\big)
 \end{align}
on $N$ and projects to the map from $N/\Gamma$ to $N$.

Let $N'$ be the commutator subgroup of $N$.  Pick a splitting of the Lie algebra ${\cal N} = {\cal N'} \oplus {\cal N}_{0} $ of $N$ where ${\cal N'}$ the Lie algebra of $N'$.  Note that ${\cal N}_{0}$ is not a Lie algebra.  Let
 $N_{0} = \exp {\cal N}_{0}$, where $\exp$ is the exponential map.
Now we decompose $h$ as a product $h=h_{1} \cdot h_{0}$, where
$h_{0}$ takes values in $N_{0}$ and $h_{1}$ takes values in $N'$, in the following way. We take $h_0$ to be the exponential of the ${\cal N}_{0}$
component of $\exp^{-1} h$ and define $h_1= h \cdot (h_0)^{-1}$.
One can see that $h_{1} \in N'$ from the Campbell-Hausdorff formula since all brackets
are in ${\cal N'}$. Note that $h_0$  and $h_1$ project to maps from $N/\Gamma$ to $N$.

 {\bf Step 1}:
 We first show that $h_{0}$ is smooth.  Let $\bar{h} : N \to N' \backslash N$ be the
 composition  of $h$ with the projection $N \to N' \backslash N$. Note that $h_{0}$
 is smooth precisely when $\bar{h}$ is smooth, since by construction $\exp^{-1} h_{0}$
 and $\exp^{-1} \bar{h}$ are just related by the identification of $ {\cal N}_{0}$ with
 the Lie algebra of $N' \backslash N$.   Write the group operation in $N' \backslash N$ additively.  Denote by $\bar{A}$ the induced automorphism of $N' \backslash N$.  Then we get
\[ (I+\bar{h}) (a(x)) = \bar{A}\,(I+\bar{h}) (x).\]
Now we can use exactly the same arguments as in Section \ref{soln-cohomology} and in particular exponential mixing  to show that $\bar{h}$ is smooth.

{\bf Step 2}: We write out Equation \ref{cocycleB} in terms of the decomposition $h=h_{1} \cdot h_{0}$:

\begin{align*}
 h_{1}(a(x)) h_{0}(a(x)) a(x)=  A(h_{1}(x)) A (h_{0}(x) )   A(x)
\end{align*}
This gives the formula
\begin{align*}
h_{1} (x) =  A^{-1} (h_{1}(a(x))) A^{-1}  (h_{0} (a(x))) A^{-1} (a(x))   x^{-1}   h_{0}(x) ^{-1}
\end{align*}
Since the automorphism $A$ leaves $N'$ invariant it follows that both $h_{1} (x)$ and $A^{-1}  (h_{1} (a(x)))$ belong to $N'$.  Hence the function $Q_{1} (x) := A^{-1}  (h_{0} (a(x))) A^{-1} (a(x))   x^{-1}   h_{0}(x) ^{-1}      $
also takes values in $N'$.  In addition, $Q_{1}(x)$ is smooth by construction and satisfies the functional equation
\begin{align*}
h_{1} (x) =A^{-1}  (h_{1} (a(x)))  Q_{1}(x).
\end{align*}
Since $h_1$ project to a map from $N/\Gamma$
then so do $A^{-1}  (h_{1} (a(x)))$ and, from the equation, $Q_{1}(x)$.
Thus the equation holds in $C^0(N/\Gamma, N')$.

Now mod out by the second derived group $N''$, and denote the projected maps by bars.  Again we write multiplication in $N''\backslash N'$ additively to get
\begin{align*}
\bar{h}_{1} (x) =  (\overline{A\mid _{N'}})^{-1}  (\bar{h}_{1} (a(x)))  +  \overline{Q}_{1}(x)
\end{align*}
We can analyze the solution to this equation once again using the methods from the basic toral case, and in particular exponential mixing and uniqueness of solutions.  We conclude that $\bar{h}_{1}$ is a smooth function.
Continue this analysis by decomposing $N'$ in terms of $N''$ and a complement $N_{1}$ to $N''$ inside $N'$.  Since the series of commutator maps terminates of a nilpotent Lie group, we see that
$h$ is a smooth function.

\section{Wavefront sets}
\label{fromrauch}

We establish regularity properties of a distribution whose derivatives along a foliation ${\cal F}$ are dual to H\"{o}lder functions in a suitable fashion.   While the definitions and concepts will be developed for foliations, the proof will be entirely local on an open subset of $R ^{n_1} \times  R ^{n_2}$ and only use partial derivatives along the second factor.   However, it will be important to develop the appropriate notions for foliations for our application to the conjugacy problem in the main part of the paper.

The main theorem is a variation of results of Rauch and Taylor in \cite{RT} who assume that derivatives of the distribution along a foliation belong to various function spaces.
The novelty  here  is that the derivatives are allowed to be distributions, of a precise order less than 0.  While we only deal with the particular case of distributions dual  to certain H\"{o}lder functions, we expect this to be true much more generally.

We first lay out our assumptions on the foliation.  Let $x$ and $y$ denote the coordinates of the first and second factor of a point in $R ^{n_1} \times  R ^{n_2}$.  Suppose $z=\Gamma (x,y)$  is a  bi-H\"{o}lder homeomorphism of an open subset  $O \subset R _x ^{n_1} \times  R _y ^{n_2}$ into $\R ^{n_1 + n_2}$
with the property that $\Gamma$ has  $y$-derivatives of all orders and these derivatives  are  H\"older in $(x,y)$.  We further assume that
for fixed $x$, $\Gamma(x, _-)$  is an immersion on each $ \{x\} \times R _y ^{n_2} $.  Then we call $\Gamma$ a {\em foliation chart}, or more precisely, a {\em  H\"{o}lder   foliation chart with smooth leaves}.   On a manifold, H\"{o}lder    foliations ${\cal F}$ with smooth leaves  are defined by patching foliation charts.  If ${\cal F}$ can be defined by using smooth  foliation charts $\Gamma$, we call ${\cal F}$ {\em smooth}.
  Note that the $x \times R^{n_2}$ for $x \in \R ^{n_1}$ define a smooth foliation ${\cal Y}$ of $\R ^{n_1 + n_2}$.

We will further assume ${\cal F}$ is {\em  strongly absolutely continuous}, i.e. there is a continuous function $J(x,y) >0$ such that  all $y$-derivatives of $J$ exist and are H\"{o}lder in $x$ and $y$ and such that  for any compactly supported continuous function $u$ on $\Gamma (O)$
\[ \int u(z) dz = \int u(\Gamma(x,y)) J(x,y) dx dy .\]
Note that if a function $u(z)$ has partial derivatives along the foliation ${\cal F}$, then $ u \circ \Gamma (x,y)$ has partial $y$-derivatives.  In addition,  the  dependence of these latter derivatives  on $x$ is continuous or H\"{o}lder if the partial derivatives of $u$ along ${\cal F}$ are continuous or H\"{o}lder.
Thus the partials  $\partial_y^{\beta}(u(\Gamma(x,y))$ are well-defined, and it makes sense to discuss their regularity.

 We will now define derivatives along the foliation ${\cal F}$ on a manifold $M$ defined by foliation charts  $\Gamma$.   Fix a  standard basis for  $R _y ^{n_2}$, parallel translate it over $\R^{n_1 + n_2}$  and consider the push forward under $\Gamma$.
 This defines vector fields $V_j$ tangent to ${\cal F}$ which are smooth along the leaves of ${\cal F}$ and
whose derivatives along ${\cal F}$ of any order are H\"{o}lder  transversely to ${\cal F}$.
We say that a function $f$ has derivatives of order up to $k$ along ${\cal F}$ if for any sequence $V_{j_1}, \ldots, V_{j_k}$ the derivatives $V_{j_1} \ldots V_{j_k} (f)$ exist.  If $M$ is endowed with a Riemannian metric, equivalently we can require the following: consider any smooth vector fields $X_1, \ldots, X_k$ on $M$, and denote their orthogonal projections to the tangent spaces of ${\cal F}$ by $Z_1, \ldots, Z_k$.   Then $f$ has derivatives up to order $k$  along ${\cal F}$ if  the derivatives $Z_1 \ldots Z_k  (f)$ exist.

 \begin{lemma}
 Under the above assumptions, the derivatives of $\Gamma ^{-1}$
 along ${\cal F}$ also H\"{o}lder.
\end{lemma}

\begin{proof}  This follows from the standard formulas for differentiating the inverse of immersions, and the assumptions on H\"{o}lderness of $\Gamma$ and its derivatives along ${\cal Y}$.
Note that the correspondence of the H\"{o}lder   coefficients, while  complicated, is explicit.
\end{proof}

In our main theorem below, we will allow  the H\"{o}lder exponents of the higher order derivatives of both $\Gamma$ and $J$ to get worse with the order.
 In the following we will use a fixed non-increasing sequence $\alpha _k$ such that all ${\cal Y}$ or ${\cal F}$ derivatives of both $\Gamma$, $\Gamma ^{-1}$ and $J$ of order at most $k$ are H\"{o}lder with H\"{o}lder exponent $\alpha _k$. This is possible by the last lemma.    Note that the vectorfields $V_j$ defined above and their  derivatives along ${\cal F}$ up to order $k$ depend $\alpha _k$-H\"{o}lder  transversely to ${\cal F}$.

Fix a Riemannian metric on $M$.  Next, we introduce the space $C^{\alpha,k} _{\cal F}$ of compactly supported $\alpha $-H\"{o}lder functions on $M$ which in addition have derivatives along ${\cal F}$ of  all orders $\leq k$ and all such derivatives  are $\alpha $-H\"{o}lder  as functions on $M$.   Then  $C^{\alpha,k} _{\cal F}$ is a Banach space
 with the norm given by  the finite sequence of
$\alpha$-H\"{o}lder norms of the  derivatives along ${\cal F}$ of order
$\leq k$.  If $M$ is compact, the norm is independent of the Riemannian metric chosen up to bi-Lipschitz equivalence.
Note that $C^{\alpha,k} _{\cal F}$ is closed under multiplication.
We let $(C^{\alpha,k} _{\cal F})^*$ be the dual space to $C^{\alpha,k} _{\cal F}$.  Note that any compactly supported smooth function on $M$ naturally belongs to any $C^{\alpha,k} _{\cal F}$.  Hence any element in $(C^{\alpha,k} _{\cal F})^*$ defines a distribution on smooth functions on $M$.  Alternatively, $(C^{\alpha,k} _{\cal F})^*$ is the space of distributions (dual to smooth functions) which extend to continuous
 linear functionals on  $C^{\alpha,k} _{\cal F}$.
 As for notation, we will also write the pairing  $D(\phi) = \langle D, \phi \rangle$ for $D \in (C^{\alpha,k} _{\cal F})^*$ and $\phi \in C^{\alpha,k} _{\cal F}$.
 All of these notions apply to the special case of ${\cal F} = {\cal Y}$.

We will work with a foliation chart $\Gamma$ and use the above notation for the case $M= \Gamma (O)$.

\begin{lemma}  \label{lemma:pullback}
Under composition with $\Gamma$, functions in $C^{\alpha,k} _{\cal F}$ pull back to  functions in $C^{\alpha \alpha _k,k} _{\cal Y}$.
Conversely,    functions in $C^{\beta ,k} _{\cal Y}$ pull back to functions in $C^{\beta \alpha _k,k} _{\cal F}$ under composition with $\Gamma ^{-1}$.   In consequence, we can also pull back distributions in $(C^{\beta \alpha _k ,k} _{\cal F}) ^*$ by $\Gamma$ to get distributions in $(C^{\beta ,k} _{\cal Y}) ^*$.
\end{lemma}

\begin{proof}   Both assertions are standard, and   follow simply from the fact that  H\"{o}lder exponents multiply under composition, and don't change under addition and multiplication.
The last statement is obtained by taking duals.
 The pull back for distributions means push
forward by $\Gamma ^{-1}$.
\end{proof}

 Now we  define distributional derivatives.  Let us first consider partial derivatives  along $y$-directions for the ${\cal Y}$ foliation.
 These are the derivatives we will use in the proof of the main theorem below.
Fix a  standard basis for  $R _y ^{n_2}$, parallel translate it over $\R^{n_1 + n_2}$.
Then  the $\frac{\partial}{\partial y_i}$ derivative of a distribution $D \in (C^{\alpha,k}_{\cal Y})^*$ is defined by
evaluating on $h \in C^{\alpha,k+1} _{\cal F}$ via
\begin{align} \label{equation:der}
 \langle \frac{\partial}{\partial y_i}(D), h \rangle = - \langle D, \frac{\partial}{\partial y_i} (h) \rangle.
 \end{align}
Note that  $\frac{\partial}{\partial y_i} (D)$ is only defined on         $C^{\alpha,k+1}_{\cal Y}$, and hence,    $\frac{\partial}{\partial y_i}(D) \in (C^{\alpha,k+1}_{\cal Y})^*$.

Similarly, we  define distributional derivatives along ${\cal F}$.  Fix a  standard basis for  $R _y ^{n_2}$, parallel translate it over $\R^{n_1 + n_2}$  and consider the push forward under $\Gamma$.  This defines vector fields $V_j$ tangent to ${\cal F}$ which are smooth along the leaves of ${\cal F}$ and
 whose derivatives along ${\cal F}$ of order up to $k$ depend
$\alpha _k$-H\"{o}lder transversely for $\alpha _k$ as above.
 Assume in the following  that $\alpha \leq \alpha _k$.   Indeed the $V_i (h)$ involve the coefficients of $\Gamma$, and this  assumption will insure that taking derivatives along the $V_j$ does not affect H\"{o}lder exponents.   More precisely we have  $V_i(h) \in  C^{\alpha,k} _{\cal F}$ for $h \in C^{\alpha,k+1} _{\cal F}$ as the $V_i$ are $\alpha$-H\"{o}lder by assumption on $\alpha$.  Hence we can define  the derivative of a distribution $D \in (C^{\alpha,k}_{\cal F})^*$ by evaluating on $h \in C^{\alpha,k+1} _{\cal F}$ via
\begin{align} \label{equation:derivative}
\langle V_i (D), h \rangle = - \langle D, V_i (h) \rangle.
\end{align}
Note that  $V_i(D)$ is only defined on         $C^{\alpha,k+1}_{\cal F}$, and hence,    $V_i(D) \in (C^{\alpha,k+1}_{\cal F})^*$.

Note that pulling back derivatives $V_j (D)$ gives us $\frac{\partial}{\partial y_j}$ derivates of the pull back of $D$ on the appropriate function spaces.

Further define  $g D$ for $g \in C^{\alpha,k}_{\cal F}$ and $D \in (C^{\alpha,k} _{\cal F})^*$ by evaluating on a test function $\phi \in C^{\alpha,k}_{\cal F}$ by
\begin{align}\label{equation:product}
(g D) (\phi)= \langle g D, \phi \rangle = \langle D , g \phi \rangle
. \end{align}

We conclude that $g D \in (C^{\alpha,k} _{\cal F})^*$.
If  $D$ is given by integration against a compactly supported $L^1$- function $u$, then $g D$ is given by integrating against $g u$.

\begin{Lemma}
\label{Lemma:Leibniz}
Let  $\alpha \leq \alpha _k$ and suppose that  $g \in C^{\alpha,k+1}_{\cal F}$, and $D \in (C^{\alpha,k} _{\cal F})^*$.  Then $V_i (g \: D) = V_i (g) \: D + g V_i (D) $ holds true in  $(C^{\alpha,k+1} _{\cal F})^*$, i.e. as functionals on
$C^{\alpha,k+1}_{\cal F} $.
\end{Lemma}

\begin{proof}   We check this by evaluating both sides on $\phi \in C^{\alpha,k+1}_{\cal F}$:
\begin{align}
  \langle V_i (g \: D),\phi \rangle = -  \langle g \: D, V_i \phi \rangle= - \langle D ,g \: (V_i \phi) \rangle= - \langle D, V_i (g \phi)  - (V_i g) \phi \rangle= \notag\\
                 \langle D,  (V_i g) \phi \rangle - \langle D, V_i (g \phi) \rangle
=   \langle (V_i g) D, \phi \rangle     +   \langle V_i D, g\: \phi \rangle
=\langle (V_i g )D, \phi \rangle + \langle g (V_i D), \phi \rangle. \notag
\end{align}
\end{proof}

{\em Note:}  The inner product $ \langle D,  (V_i g) \phi \rangle$ is not defined unless
$g \in C^{\alpha,k+1}_{\cal F}$.  Thus we need the higher regularity on  $g$ in the hypothesis of
the previous lemma.  This simple problem caused the introduction of the spaces of test functions
$C^{\alpha,k}_{\cal F}$.

Let  $u$ be an $L^1$ function defined on a neighborhood of a point $z_0$.
A vector $\zeta _0$ is called {\em not singular}  for $u$ at $z_0$  if
there exist an open set $\mathcal U \ni  z_0$ and an open cone
$\mathcal Z \subset \R^n \setminus\{0\}$ around $\zeta _0$ such
that for any positive integer $N$ and any $C^{\infty}$ function $\chi$
with support in $\mathcal U$ there exists a constant $C=C(N,\chi)$ so that
\begin{equation}
\label{WF}
|\widehat{\chi u}\, (\zeta)|= \big| \int u(z) \chi (z) \exp(-iz \cdot \zeta)  dz \big|
\leq C|\zeta|^{-N} \quad \text{ for all } \zeta \in \mathcal Z \text{ with }
|\zeta|>1.
\end{equation}
Otherwise, $\zeta _0$ is called {\em singular} for $u$ at $z_0$.
The {\em wave front set} $WF(u)$ is defined as the set of all
$(z_0,\zeta _0)$ such that $\zeta _0 $ is singular for $u$ at $z_0$.

\begin{theorem}
\label{theorem:rauchlike} Suppose that $u(z)$ is an $L^1$ function.
Let  ${\cal F}$ be a H\"{o}lder foliation with smooth leaves which is
also strongly absolutely continuous.  Consider  the distribution $D$
defined by integration against $u (z)$.
Assume  that any
derivative of $D$ along ${\cal F}$ of any order
belongs to $(C^{\alpha}_{\cal F})^*$ for all positive $\alpha$.
If $(z_0,\zeta _0) \in T^*(\R^n)\backslash{0}$ is not conormal to
$\mathcal F$ then
$$(z_0, \zeta _0) \notin WF(u).$$
\end{theorem}

As an immediate corollary, we obtain the result needed in
Section \ref{soln-cohomology}.

\begin{corollary} \label{what we need}
Let  ${\cal F}_1, \ldots, {\cal F}_r$ be H\"{o}lder   foliations with smooth leaves on a  manifold $M$ which are also strongly absolutely continuous.  Assume in addition that the tangent spaces to these foliations span the tangent spaces to $M$ at all points.

Now suppose that $u(z)$ is an $L^1$ function.  Consider  the distribution $D$ defined by integration against $u (z) $.   Assume that any derivative of $D$ of any order along  any ${\cal F}_i, i= 1, \ldots, r$ belongs to $(C^{\alpha}_{{\cal F}_i})^*$ for all $1 \leq i \leq r$ and all positive $\alpha$.
Then $u$ is $\Ci$.
\end{corollary}

\begin{proof} Since the tangent spaces to the foliations span the tangent bundle everywhere, no vector $\zeta \neq 0$ can be conormal to all ${\cal F}_i$.
 Now it follows from Theorem \ref{theorem:main}  that $WF(u)$ is empty
 and hence $u$ is smooth by e.g. \cite[Section 8.1]{Hormander}.
\end{proof}

The main idea in the proof of Theorem \ref{theorem:rauchlike} is a simple
generalization of an argument of Rauch and Taylor in \cite{RT}.
However much more care has to be taken to make sure that various
operations undertaken are well defined and allowed.  In particular,
we use integration by parts for derivatives along the foliation.
This requires that the test functions in question are differentiable
along ${\cal F}$ up to a suitable order.  This led to the definition of
 the function spaces above.

{\em Remark:} The proof of Theorem \ref{theorem:main} becomes
easier if the foliation ${\cal F}$ has derivatives of all orders  of a fixed
H\"{o}lder class and the distribution in question together with its
derivatives along ${\cal F}$ are dual to a fixed H\"{o}lder class.

\begin{proof}
We fix $(z_0,\zeta _0)$ which in not conormal to $\cal F$.
By the definition of the wave front set it suffices to show that there
exist an open set $\mathcal U \ni  z_0$ and an open cone
$\mathcal Z \subset \R^n \setminus\{0\}$ around $\zeta _0$ such
that for any $N>0$ and any $\chi \in C^{\infty}_0(\mathcal U)$
there exists a constant $C$ so that
\begin{equation}
\label{equation:toshow}
| \widehat{\chi u}\, (\zeta)|=\big| \int u(z) \chi (z) \exp(-iz \cdot \zeta) dz \big|
\leq C|\zeta|^{-N} \quad \text{ for all } \zeta \in \mathcal Z
\text{ with } |\zeta|>1.
\end{equation}
We define
$$\phi(x,y,\zeta)= -\Gamma(x,y)\cdot \zeta,$$
and note that, for a fixed $\zeta$, the function $\phi$ is in ${\mathcal C} ^{\alpha _k,k}$ for all $k$ by the choice of $\alpha _k$.
 Using a foliation
chart and the strong absolute continuity of $\cal F$ we can write
$$
\widehat{\chi u}\, (\zeta) = \int u(\Gamma(x,y)) \:
\chi (\Gamma (x,y)) J(x,y)\exp(i\phi(x,y,\zeta)dxdy.
$$
The hypotheses that $(z_0, \zeta _0)$ is not conormal to $\mathcal F$
implies that
$$ d_y\phi(x,y, \zeta _0) \neq 0, \text{  where  } \Gamma(x,y)=z.$$
Relabeling the $y$ coordinates it follows that there exist a
neighborhood $\mathcal U$ of $z _0$, an open cone
$\mathcal Z \subset \R^n \setminus\{0\}$ around $\zeta _0$, and
$\delta >0$ so that

\begin{equation}
\label{equation:togetdecay}
\left| \frac{\partial\phi(x,y,\zeta)}{\partial y_1} \right| > \delta |\zeta|
\text{,  when   }(\Gamma(x,y),\zeta) \in  \mathcal U \times \mathcal Z
\end{equation}
To obtain the desired decay in $\zeta$ we use the identity
$$\left( \frac{1}{i\partial\phi(x,y,\zeta)/\partial y_1}\frac{\partial}{\partial
y_1}\right) \exp(i\phi(x,y,\zeta)= \exp(i\phi(x,y,\zeta))$$
to deduce that
\begin{equation}
\label{equation:beforeexpand}
\widehat{\chi u} (\zeta)=
\int u(\Gamma(x,y))\chi(\Gamma(x,y))J(x,y)\left( \frac{1}{i\partial\phi(x,y,\zeta)/\partial y_1}\frac{\partial}{\partial
y_1}\right) ^N\exp(i\phi(x,y,\zeta)dxdy.
\end{equation}
We can expand
\begin{equation}
\label{equation:expand}
 \bigg(\frac{1}{i\partial\phi(x,y,\zeta)/\partial y_1}\frac{\partial}{\partial
y_1}
\bigg)^N
\ =\
\sum_{m=1}^N
\psi_{m,N}(x,y,\zeta)
\
\bigg(\frac{\partial}{\partial y_1}\bigg)^m.
\end{equation}

To describe functions $\psi_{m,N}(x,y,\zeta)$ we note that
$(g\frac{\partial}{\partial y_1})^N$ is a sum of terms  of the form
$P_m(\frac{\partial}{\partial y_1})^m$, where $P_m$ is a polynomial
in $g$ and its first $(N-m)$ derivatives. Applying this to
$g=\frac{1}{i\partial\phi(x,y,\zeta)/\partial y_1}$, we see that
each function $\psi_{m,N}(x,y,\zeta)$ is a quotient of a polynomial
in $\Gamma(x,y)\cdot \zeta$ and its first $(N-m+1)$ derivatives divided
by a power of ${i\partial\phi(x,y,\zeta)/\partial y_1}$. Taking $k$
derivatives of $\psi_{m,N}$ yields, by the product and quotient rules,
a similar expression which involves derivatives of $\Gamma(x,y)$
of order $(N-m+1+k)$ and hence is H\"older with exponent
$\a _{(k+N-m+1)}$. It follows that, for a fixed $\zeta$ and any $m=1 ,..., N$,
the function $\psi_{m,N}(x,y,\zeta)$ is in
$C^{\alpha_{(N+1)},\, m}_{\mathcal Y}$. Moreover, there exists a
constant $C$ such that
\begin{equation}
\label{equation:moreondecay}
\| \psi_{m,N}
\|_{\alpha_{(N+1)},m} \le \, C \,|\zeta|^{-N}
\quad \text{ for all } \zeta \in \mathcal Z \text{ with } |\zeta|>1.
\end{equation}
Indeed, since $\phi(x,y,\zeta)$ is linear in $\zeta$, both
sides of \eqref{equation:expand} are homogeneous of degree $-N$
in $\zeta$, and hence so are the functions $\psi_{m,N}$ and their
derivatives. We conclude that the functions in
\eqref{equation:moreondecay} are rational functions in $\zeta$ of
homogeneous degree $-N$ whose coefficients, as functions of $(x,y)$,
are  H\"older  on $\Gamma^{-1} (\mathcal U)$.
The H\"older norms of these coefficients are continuous in
$\zeta$ and hence are uniformly bounded on
$\mathcal Z \cap \{ |\zeta|=1 \}$. Finally, using equation
\eqref{equation:togetdecay} we can
bound the denominators away from zero and obtain
\eqref{equation:moreondecay}.
\vskip.3cm

Using \eqref{equation:beforeexpand} and \eqref{equation:expand}
we can write $\widehat{\chi u}$ as a finite sum
$$
 \widehat{\chi u} (\zeta)= \sum_{m=1}^N   \int
u(\Gamma(x,y))\,
\chi(\Gamma(x,y))\,
J(x,y)\,
\psi_{m,N}(x,y,\zeta)
\bigg(\frac{\partial}{\partial y_1}\bigg)^m
\exp(i\phi(x,y,\zeta)) dxdy.
$$
In the remainder of the proof we estimate each term of this sum. For this we denote
$$
A=u(\Gamma(x,y))\,\chi(\Gamma(x,y)) \quad \text{and} \quad
A_{m,N} ^{\zeta}=u(\Gamma(x,y))\, \chi(\Gamma(x,y))\, J(x,y)\,\psi_{m,N}(x,y,\zeta)
$$
and view $A$ and $A_{m,N} ^{\zeta}$ as the distributions
given by integration, with a fixed $\zeta$, against the corresponding
functions. Since the functions $u \circ \Gamma, \chi \circ \Gamma, J$ and $\psi _{M,n}$ are in $L^1$, $A$ and
$A_{m,N}^{\zeta}$ lie in
$({\mathcal C}^{\alpha} _{{\cal Y}})^*=({\mathcal C}^{\alpha,0} _{{\cal Y}})^*$
for all positive $\alpha$, and
$A_{m,N}^{\zeta} =  J \psi_{m,N} A$ as elements of
$({\mathcal C}^{\alpha}_{\cal Y})^*$ with multiplication of
distributions defined as in equation \eqref{equation:product}.
Recall that $\phi (x,y,\zeta)$ is in  $(C^{\alpha _{m} ,m}_{\cal Y})$,
so by the definition of derivatives of distributions
for each term in $ \widehat{\chi u} (\zeta)$ we obtain

\begin{align}
\notag
\int   u(\Gamma(x,y))\,    \chi(\Gamma(x,y))\,   J(x,y)\,     \psi_{m,N}(x,y,\zeta) \
\bigg(\frac{\partial}{\partial y_1}\bigg)^m  \exp(i\phi (x,y,\zeta))  \ dxdy  \\
= \langle  A_{m,N} ^{\zeta},\bigg(\ \bigg(\frac{\partial}{\partial y_1}\bigg)^m  \exp(i\phi (x,y,\zeta)\bigg)\rangle
=  (-1) ^m \langle \bigg(\ \bigg(\frac{\partial}{\partial y_1}\bigg)^m (A_{m,N}^{\zeta}) \bigg),  \exp(i\phi (x,y,\zeta)\rangle  \,  \notag \\
\notag =  (-1) ^m \langle \bigg(\ \bigg(\frac{\partial}{\partial y_1}\bigg)^m (J \psi_{m,N} A) \bigg) ,  \exp(i\phi (x,y,\zeta)\rangle  \,  \notag ,
\end{align}
where the pairing is  in the sense of
$(C^{\alpha _{m} ,m}_{\cal Y})^*$ for  $1 \leq m \leq N$.
Now we apply the Leibniz rule, Lemma \ref{Lemma:Leibniz},
$m$ times to write
$\left(\frac{\partial}{\partial y_1}\right)^m (J \psi_{m,N} A)$
as
\begin{align}   \label{equation:last}
 \bigg(\frac{\partial}{\partial y_1}\bigg)^m (J \psi_{m,N} A) =
 \sum_{a+b+c = m}   K_{a,b,c}
\bigg[\left(\frac{\partial}{\partial y_1}\right)^a
J(x,y)\bigg]\,
\bigg[\left(\frac{\partial}{\partial y_1}\right)^b
(\psi_{m,N})\bigg]\,\bigg[ \left(\frac{\partial}{\partial y_1}\right)^c
A  \bigg]
  \notag.
\end{align}
The equation holds in $(C^{\alpha_{N+1},m}_{\cal Y})^*$ since $A$
is in $({\mathcal C}^{\alpha_{N+1}}_{\cal Y})^*$ and $\psi_{m,N}$
as well as $J$ are in $C^{\alpha_{N+1},m}_{\cal Y}$.
Finally, we can rewrite the pairing in $(C^{\alpha_{N+1},m}_{\cal Y})^*$
of each term in this sum with $\exp(i\phi (x,y,\zeta)$ as

$$
\langle \bigg(\bigg[\left(\frac{\partial}{\partial y_1}\right)^a
J(x,y)\bigg] \cdot
\bigg[\left(\frac{\partial}{\partial y_1}\right)^b
(\psi_{m,N})\bigg] \cdot \bigg[ \left(\frac{\partial}{\partial y_1}\right)^c
A  \bigg] \bigg) ,  \exp(i\phi (x,y,\zeta)\rangle =
$$
\begin{equation}
\label{equation:decayterm}
\langle  \left(\frac{\partial}{\partial y_1}\right)^c
A  \, , \bigg(\bigg[\left(\frac{\partial}{\partial y_1}\right)^a
J(x,y)\bigg] \cdot
\bigg[\left(\frac{\partial}{\partial y_1}\right)^b
(\psi_{m,N})\bigg] \cdot  \exp(i\phi (x,y,\zeta)\bigg)\rangle .
\end {equation}
\vskip.3cm

Now we use the assumption that derivatives of $u$ and hence of
the localization $u \chi$ along ${\mathcal F}$ exist as elements
in $({\mathcal C}^{\alpha}_{\cal F})^*$ for all positive $\alpha$.
Therefore, by Lemma \ref{lemma:pullback}, $y$-derivatives of the
pull back $A=(u \chi)\circ \Gamma$ also exist as elements in
$({\mathcal C}^{\alpha}_{\cal Y})^*$ for all positive $\alpha$.
Hence the pairing in \eqref{equation:decayterm} can be estimated
by $(\alpha_{N+1})$-H\"older norm
of the product  $\big[
\big(\frac{\partial}{\partial y_1}\big)^a J(x,y) \big] \big[
\big(\frac{\partial}{\partial y_1}\big)^b  \: (\psi_{m,N}) \big]
\exp(i\phi (x,y,\zeta))$. As $b,c \le m$, all three functions are
$(\alpha_{N+1})$-H\"older.
Moreover, for all $\zeta \in \mathcal Z$  with $|\zeta|>1$,
 $\| \big( \frac{\partial}{\partial y_1}\big)^a J(x,y) \| _{\alpha_{N+1}}$
is bounded by a fixed constant,
$\| \big(\frac{\partial}{\partial y_1}\big)^b  \: (\psi_{m,N}) \| _{\alpha_{N+1}}
\le \, C \,|\zeta|^{-N}$ by \eqref{equation:moreondecay}, and
the norm $\| \exp(i\phi (x,y,\zeta)) \|_{\alpha_{N+1}}$ can be
estimated by $C' \,|\zeta|$.
We conclude that each pairing in \eqref{equation:decayterm}
can be estimated by $C'' \,|\zeta|^{-N+1}$, and hence the same
estimate holds for $|\widehat{\chi u} (\zeta)|$. Since $N$
is arbitrary, the desired estimate \eqref{equation:toshow} now
follows and shows that any $(z_0,\zeta _0)$ which is not conormal
to $\cal F$ is not the wave front set of $u$.
\end{proof}

\vspace{1cm}

\newcommand{\cf}{{\cal F}}
\newcommand{\cp}{{\cal P}}
\newcommand{\cA}{{\mathcal{A}}}
\newcommand{\cB}{{\mathcal{B}}}
\newcommand{\cC}{{\mathcal{C}}}
\newcommand{\cD}{{\mathcal{D}}}
\newcommand{\cE}{{\mathcal{E}}}
\newcommand{\cL}{{\mathcal{L}}}
\newcommand{\cM}{{\mathcal{M}}}
\newcommand{\cO}{{\mathcal{O}}}
\newcommand{\cR}{{\mathcal{R}}}
\newcommand{\cS}{{\mathcal{S}}}
\newcommand{\cT}{{\mathcal{T}}}
\newcommand{\cU}{{\mathcal{U}}}
\newcommand{\cV}{{\mathcal{V}}}
\newcommand{\bfH}{{\mathbf H}}
\newcommand{\bfK}{{\mathbf K}}
\newcommand{\bfS}{{\mathbf S}}

\newcommand{\F}{{\mathbb F}}
\newcommand{\wh}[1]{{\widehat{#1}}}
\newcommand{\wt}[1]{{\widetilde{#1}}}

\appendix  \label{appendix}
\section{ }
\vspace{.5em}

\begin{center}{\bf A FINITE COVER OF AN EXOTIC NILMANIFOLD IS STANDARD}

\vspace{.7em}
BY JAMES   F. DAVIS
\vspace{1em}
\end{center}

A {\em nilmanifold} is the quotient $G/L$ of a simply connected nilpotent Lie group $G$ by a discrete cocompact subgroup $L$.  Two homeomorphisms $f,g : X \to Y$ are {\em isotopic} if they are homotopic through homeomorphisms.

\begin{theorem}  \label{nil_smoothing}
Let $h : M \to G/L$ be a homeomorphism from a smooth manifold to a nilmanifold of dimension greater than four.  Then there is a finite cover $\widehat{G/L} \to G/L$ so that the induced pullback homeomorphism $\wh M \to \widehat{G/L}$ is isotopic to a diffeomorphism.
\end{theorem}

Theorem \ref{nil_smoothing} is a consequence of Lemma \ref{general} and Lemma \ref{nil_*} stated below.

 \begin{definition}  A space $N$ satisfies condition (*) if for any $i > 0$, for any finite abelian group $T$, for any finite cover $\hat p: \hat N \to N$, and for any $x \in H^i(\hat N; T)$, then there exists a finite cover $\tilde p : \tilde N \to \hat N$ so that $\tilde p^*x = 0$.
\end{definition}

\begin{lemma} \label{general}
Let $h : M \to N$ be a homeomorphism of smooth manifolds of dimension greater than four.  {\colb Suppose $N$} satisfies (*).  Then there is a finite cover $\widehat{N} \to N$ so that the induced pullback homeomorphism $\widehat{M} \to \widehat{N}$ is isotopic to a diffeomorphism.
\end{lemma}

In particular any two smooth structures on $N$ become diffeomorphic after passing to a finite cover.  An existence result can be proved using similar techniques: any topological manifold of dimension greater than four which satisfies (*) has a finite cover which admits a smooth structure.

\begin{lemma} \label{nil_*}
Any nilmanifold satisfies condition (*).
\end{lemma}

\begin{proof}
Since a finite cover of a nilmanifold is a nilmanifold, it will be notationally simpler to show that any nilmanifold satisfies condition (**) defined below.

A space $N$ satisfies condition (**) if for any $i > 0$, for any finite abelian group $T$, and for any $x \in H^i(N; T)$, then there exists a finite cover $\widehat{p} :  \wh N \to  N$ so that $\wh p^*x = 0$.

We first verify condition (**) when $i = 1$.  Indeed, the Universal Coefficient Theorem gives an isomorphism $H^1(N; T) \to \Hom(H_1(N); T)$ for all spaces $N$ and the Hurewicz Theorem gives an isomorphism $\pi_1(N,n_0)^\text{ab} \to H_1(N)$ for a path-connected  space $N$.  Thus there is a natural isomorphism of contravariant functors from path-connected based spaces to abelian groups
$$
\Phi(N,n_0) : H^1(N; T) \xrightarrow{\cong} \Hom(\pi_1(N,n_0),T).
$$

Given $x \in H^1(N;T)$, there is a connected cover $\wh p : \wh N \to N$ and a base point $\wh n_0 \in \wh N$ so that $$ \wh p_*(\pi_1(\wh N, \wh n_0)) = \ker (\Phi(N,n_0)(x) : \pi_1(N,n_0) \to T).$$
Since $T$ is a finite group, $\wh p$ is a finite cover .  The commutative square
$$
\begin{CD}
H^1(\wh N; T) @>\cong>> \Hom(\pi_1(\wh N,\wh n_0),T)\\
@AA\wh{p}^*A @AA   - ~\circ ~\wh p_*A \\
H^1(N; T) @>\cong>> \Hom(\pi_1(N,n_0),T)
\end{CD}
$$
shows that $\wh p^*x = 0$.

We now turn to the proof that any nilmanifold satisfies condition (**) when $i > 1$.  The proof will be by induction on the dimension of the nilmanifold $N = G/L$, using the Gysin sequence of a principal $S^1$-bundle
$$
S^1 \to N \xrightarrow{\pi} N/S^1
$$
where $N/S^1$ is a nilmanifold.  To obtain this principal bundle note that the center $Z(G)$ is nontrivial since $G$ is nilpotent.  Furthermore, it can by shown that $Z(L) = L \cap Z(G)$ is a discrete cocompact subgroup of the real vector space $Z(G)$ (see \cite[Proposition 2.17]{Raghunathan}).  Choose a primitive element $l \in L \cap Z(G)$.  Then $S^1 = \R \cdot l/ \Z \cdot l$ acts freely on $N$ and the quotient $N/S^1$ is the nilmanifold $(G/\R \cdot l)/(L/ \Z \cdot l)$.

Let $N$ be a nilmanifold.  Assume by induction that condition (**) holds for all nilmanifolds of strictly smaller dimension.   The Gysin sequence (see \cite{DavisKirk})
$$
\cdots \to H^{i-2}(N/S^1; T) \xrightarrow{\cup e}  H^i(N/S^1; T) \xrightarrow{\pi^*} H^i(N; T) \xrightarrow{\pi_!} H^{i-1}(N/S^1; T) \to \cdots
$$
is an exact sequence associated to a principal $S^1$-fibration.   By the inductive hypothesis, there exists a finite cover $\wt{p/S^1} : \wt{N/S^1} \to N/S^1$ so that $\wt{p/S^1}^*(\pi_! x) = 0$.  (Note, here is where we use that $i > 1$.)  Define $\wt N$ as the pullback
$$
\begin{CD}
\wt N @>\wt \pi>>  \wt{N/S^1} \\
@VV\wt p V @VV\wt{p/S^1}V \\
N @>\pi>>  N/S^1
\end{CD}
$$
We have a map of principal $S^1$ bundles, hence a map of Gysin sequences (see the bottom two rows of the diagram below).   By commutativity of the lower right square below and the exactness of the middle row, there is an $x' \in H^i(\wt{N/S^1}; T)$ so that $\wt \pi^*x' = \wt p^* x$.  By the inductive hypothesis again, there is a finite cover $\wh{p/S^1} : \wh{N/S^1} \to \wt{N/S^1}$ so that $\wh{p/S^1}^*( x') = 0$.  Defining $\wh N$ as a pullback, we have the diagram below.

$$
\begin{CD}
H^i(\wh{N/S^1}; T) @>\wh \pi^*>> H^i(\wh N; T) @>\wh \pi_!>>  H^{i-1}(\wh{N/S^1}; T) \\
@AA\wh{p/S^1}^*A @AA\wh{p}^*A @AA\wh{p/S^1}^*A \\
H^i(\wt{N/S^1}; T) @>\wt\pi^*>> H^i(\wt N; T) @>\wt\pi_!>> H^{i-1}(\wt{N/S^1}; T) \\
@AA\wt{p/S^1}^*A @AA\wt{p}^*A @AA\wt{p/S^1}^*A \\
H^i(N/S^1; T) @>\pi^*>> H^i(N; T) @>\pi_!>>H^{i-1}(N/S^1; T)
\end{CD}
$$

Hence our desired finite cover is $\wt p \circ \wh p : \wh N \to N$.  This completes the proof of the lemma.
\end{proof}

In preparation for the proof of Lemma \ref{general} we review a bit of smoothing theory.  The two definitive treatments are the books \cite{KirbySeibenmann} and \cite{HirschMazur}; see also the recent survey \cite{DavisPetrosyan}.  A {\em smooth structure on a topological manifold $\Sigma$} is a pair $(M, h)$ where $M$ is a smooth manifold and $h : M \to \Sigma$ is a homeomorphism.  Two smooth structures $(M_1, h_1)$ and $(M_2,  h_2 )$ are {\em isotopic} if there is a diffeomorphism $f : M_1 \to M_2$ so that $h_1$ is isotopic to $h_2 \circ f$.  Let ${\cal T}_O(\Sigma)$ be the set of isotopy classes of smooth structures on $\Sigma$.

The fundamental theorem of smoothing theory says that a topological manifold of dimension greater than four admits a smooth structure if and only if its topological tangent bundle admits the structure of a vector bundle.  Furthermore, isotopy classes of smooth structures are in bijective correspondence with bundle reductions.   It will be easier (and slicker) to express this in terms of maps to classifying spaces, as in Part 2 of \cite{HirschMazur}.

Let $Top(n)$ be the group of homeomorphisms of $\R^n$ fixing the origin.  Give $Top(n)$ the compact open topology.    Let $O(n)$ be the orthogonal group.  Let $Top = \colim Top(n)$ and $O = \colim O(n)$.  The quotient space $Top/O$ admits the structure of an abelian $H$-space satisfying the following property:  if $\Sigma$ is a topological manifold of dimension greater than four, then the abelian group of homotopy classes $[\Sigma, Top/O]$ acts freely and transitively on the set of isotopy classes of smooth structures ${\cal T}_O(\Sigma)$.  For smooth structures $(M_1,h_1)$ and $(M_2,h_2)$, let $d(h_1,h_2)$ be the unique element of $[\Sigma,Top]$ so that $d(h_1,h_2) [M_1,h_1]= [M_2,h_2] \in {\cal T}_O(\Sigma)$.  Thus $d(h_1,h_2) = 0$ if and only if the homeomorphism $h_2^{-1} \circ h_1 : M_1 \to M_2$ is isotopic to a diffeomorphism.

The homotopy groups of $Top/O$ are reasonably well-understood.  Indeed, $\pi_i(Top/O) = 0,0,0,\Z/2,0,0,0,{\colb \Z/28}$ for $i = 0,1,2,3,4,5,6,7$ and for $i \geq 5$, $\pi_i(Top/O) \cong \Theta_i$, the group of exotic smooth structures on the $i$-sphere.  In particular, $Top/O$ is simply connected and the homotopy groups $\pi_i(Top/O)$ are all finite.

\begin{proof}[Proof of Lemma \ref{general}]
Let $\Sigma$ be a topological manifold of dimension greater than 4 which admits a smooth structure.  Here are three observations.  First, if $\wh p : \wh \Sigma \to \Sigma$ is a covering map, then the map $\wh p^* : {\cal T}_O(\Sigma) \to {\cal T}_O(\wh \Sigma)$ is equivariant with respect to the group homomorphism $\wh p^* : [\Sigma, Top/O] \to [\wh \Sigma, Top/O]$.  In other words, $\wh p^*([\alpha]\cdot [M,h]) = \wh p^*[\alpha]\cdot \wh p^*[M,h]$.  The geometric fact underlying this is that the pullback of the tangent bundle of the base space under a covering map is the tangent bundle of the total space.  Second, note that $\Sigma$ admits the structure of a CW-complex, for example, by triangulating the smooth structure.  Finally, note that if $f,g : X \to Y$ are maps from a  CW-complex to a simply-connected space, and $H(i-1) : X^{i-1} \times I \to Y$ is a homotopy from $f|_{X^{i-1}}$ to $g|_{X^{i-1}}$, there is a well-defined obstruction class ${\cal O} = {\cal O}^{i}(f,g,H(i-1)) \in H^{i}(X; \pi_{i}Y)$ (see \cite[Theorem 7.12]{DavisKirk}.  This class vanishes if and only if there is a homotopy $H(i) : X^i \times I \to Y$  from $f|_{X^{i}}$ to $g|_{X^{i}}$ which restricts to $H(i-1) |_{X^{i-2} \times I}$.

Let $(M_1,h_1)$ and $(M_2,h_2)$ be two smooth structures on a topological manifold $\Sigma$ which satisfies condition (*).  Assume $n = \dim \Sigma \geq 5$.  Give $\Sigma$ the structure of an $n$-dimensional CW complex.  Assume, by induction, there exists a finite cover $\wh p_{i-1} : \wh \Sigma_{i-1} \to \Sigma$ so that $d(\wh p_{i-1}h_1,\wh p_{i-1}^*h_2)$ is represented by a map $ \wh \Sigma_{i-1}  \to Top/O$ which is null-homotopic restricted to the $(i-1)$-skeleton.  Let ${\cal O} \in H^i(\wh \Sigma_{i-1} ; \pi_i(Top/O))$ be the obstruction to extending to null-homotopy.  By condition (*), there is a finite cover $\wh p(i):  \wh \Sigma_{i} \to  \wh \Sigma_{i-1}$ so that $\wh p(i)^* {\cal O} = 0$.  Then the finite cover $\wh p_i :=
\wh p_{i-1} \circ \wh p(i): \wh \Sigma_i \to \Sigma $ satisfies the inductive hypothesis.  Thus $\wh p_n$ is a finite cover so that the smooth structures $\wh p_n^* h_1$ and $\wh p_n^* h_2$ are isotopic.
\end{proof}

\begin{remark}  \label{non-exotic}
Suppose $\Sigma$ is a manifold of dimension 3 or less.  Using the work of many mathematicians, most notably Rado and Moise, one can show (see \cite{DavisPetrosyan}) that $\Sigma$ admits a smooth structure and that any two smooth structures are isotopic.
\end{remark}


\bibliographystyle{alpha}

\end{document}